\documentclass[10pt]{article}
\makeatletter
\renewcommand{\thefootnote}{}

\usepackage{epsfig,amsmath,graphicx,amssymb,overpic}

\usepackage{graphicx}% Include figure files
\usepackage{dcolumn}% Align table columns on decimal point
\usepackage{bm}% bold math
\usepackage{graphicx}
\usepackage{subfigure}
\usepackage{epsfig,amsmath,graphicx,amssymb,overpic,cite}

\usepackage{graphicx}% Include figure files
\usepackage{dcolumn}% Align table columns on decimal point
\usepackage{bm}% bold math
\usepackage{graphicx}
\usepackage{subfigure}
\usepackage{amsthm}

\usepackage{pdfsync}%pdf反向搜索用到的包
\usepackage{cite}
\usepackage{amscd}
\usepackage{amsmath}
\usepackage{latexsym}
\usepackage{amsfonts}
\usepackage{amssymb}
\usepackage{amsthm}
\usepackage{graphicx}
\usepackage{indentfirst}
\usepackage{enumerate}
\usepackage{amstext}
\usepackage[dvipdfm,
            pdfstartview=FitH,
            CJKbookmarks=true,
            bookmarksnumbered=true,
            bookmarksopen=true,
            colorlinks=true, %注释掉此项则交叉引用为彩色边框(将colorlinks和pdfborder 同时注释掉)
            pdfborder=001,   %注释掉此项则交叉引用为彩色边框
            citecolor=magenta,% magenta , cyan
            linkcolor=blue,
            ]{hyperref}

 \allowdisplaybreaks[4]

\newtheorem{corollary}{Corollary}[section]
\newtheorem{lemma}{Lemma}[section]

\newtheorem{theorem}{Theorem}[section]

\newtheorem{defi}{\textbf Definition}[section]

\usepackage{tikz}
\usetikzlibrary{positioning} %为了实现相对位置的设定
\usepackage{xcolor} %为了实现不同的颜色

\def\be{\begin{equation}}
\def\ee{\end{equation}}
\def\bee{\begin{eqnarray}}
\def\ene{\end{eqnarray}}
\def\bes{\begin{subequations}}
\def\ees{\end{subequations}}
\def\det{{\rm det}}
\def\d{\displaystyle}
\def\v{\vspace{0.1in}}
\def\no{{\nonumber}}

\begin{document}

\baselineskip=13pt
\renewcommand {\thefootnote}{\dag}
\renewcommand {\thefootnote}{\ddag}
\renewcommand {\thefootnote}{ }

\pagestyle{plain}

\begin{center}
\baselineskip=16pt \leftline{} \vspace{-.3in} {\Large \bf Well-posedness of scattering data for the derivative nonlinear
Schr\"odinger equation in $H^s(\mathbb{R})$
} \\[0.2in]
\end{center}

\begin{center}
{\bf Weifang Weng$^{a}$, Zhenya Yan}$^{b,c,*}$\footnote{$^{*}${\it Email address}: zyyan@mmrc.iss.ac.cn (Corresponding author)}  \\[0.08in]
{\it$^a$Department of Mathematical Sciences, Tsinghua University, Beijing 100084, China} \\
$^b$School of Mathematical Sciences, University of Chinese Academy of Sciences, Beijing 100049, China \\
$^c$Key Laboratory of Mathematics Mechanization, Academy of Mathematics and Systems Science, \\ Chinese Academy of Sciences, Beijing 100190, China\\[0.18in]
 %(Date:\,\, \today)
\end{center}

\noindent {\bf Abstract:}\, {\small We prove the well-posedness results of scattering data for the derivative nonlinear Schr\"odinger equation in  $H^{s}(\mathbb{R})(s\geq\frac12)$. We show that the reciprocal of the transmission coefficient can be written as the sum of some iterative integrals, and its logarithm can be written as the sum of some connected iterative integrals. And we provide the asymptotic properties of the first few iterative integrals of the reciprocal of the transmission coefficient. Moreover, we provide some regularity properties of the reciprocal of the transmission coefficient related to scattering data in $H^{s}(\mathbb{R})$.
%We establish a family of conserved energies for the derivative nonlinear Schr\"odinger equation, such that all the energy norms of the solutions are conserved globally in time.
}

\vspace{0.1in} \noindent {\it Keywords:}\, Derivative nonlinear Schr\"odinger equation; Lax pair;
Inverse scattering transform; Transmission coefficient; Well-posedness

%\tableofcontents

%\vspace{0.1in} \noindent {\bf Mathematics Subject Classification}\, 35B30; 35G25; 35B44; 35Q35

\section{Introduction}

As an important and fundamental nonlinear mathematical and physical model, the derivative nonlinear Schr\"odinger (DNLS) equation
\bee \label{dnls}
iq_{t}+q_{xx}\pm i(|q|^2q)_x=0,\quad q=q(x,t), \quad x\in\mathbb{R},
\ene
appears in many fields,  such as the wave propagation of circular polarized nonlinear Alfv{\' e}n waves in plasmas \cite{Rogister1971,Mjolhus1976, Mio1976, Mjolhus1989, Mjolhus1997}, weak nonlinear electromagnetic waves in ferromagnetic \cite{Nakata1991}, antiferromagnetic \cite{Daniel2002} or dielectric \cite{Nakata1993} systems under external magnetic fields.
 % where the subscripts denote the partial derivatives with respect to the variables $x,\, t$ and $q(x, t)$ : $\mathbb{R}\times \mathbb{R}\rightarrow \mathbb{C}$ denotes the complex-valued function. The DNLS equation was derived by Mio-Ogino-Minami-Takeda and Mjolhus for studying the one-dimensional compressible magneto-hydrodynamic equation in the presence of the Hall effect and the propagation of circular polarized nonlinear Alfv$\acute{e}$n waves in magnetized plasmas\cite{dnls-Mio, dnls-Mjolhus}.
Without loss of generality, one can take sign $+$ (since the case sign $+$ can be transformed into sign $-$ by means of $x \rightarrow -x$). Kaup and Newell~\cite{Kaup1978} showed that Eq.~(\ref{dnls}) was completely integrable, and has the following modified Zakharov-Shabat eigenvalue problem (Lax pair)~\cite{Kaup1978}:
\begin{align}\label{lax}
\begin{aligned}
\psi_x=U(x,t,\lambda)\psi,\\
\psi_t=V(x,t,\lambda)\psi,
\end{aligned}
\end{align}
with
\begin{equation}
\begin{array}{l}
U(x,t,\lambda)=-i\sigma_3(\lambda^2+i\lambda Q),\vspace{0.1in}\\
V(x,t,\lambda)=\left(\!\!\begin{array}{cc}
-i(2\lambda^4-\lambda^2|q|^2)& 2\lambda^3q-\lambda|q|^2q+i\lambda q_x  \vspace{0.05in}\\
-2\lambda^3q^*+\lambda|q|^2q^*+i\lambda q^*_x & i(2\lambda^4-\lambda^2|q|^2)
\end{array}\!\!\right),\vspace{0.1in}\\
Q=\!\left(\!\!\begin{array}{cc}
0& q(x,t)  \vspace{0.05in}\\
q^*(x,t) & 0
\end{array}\!\!\right), \quad
\sigma_3\!=\!\left(\!\!\begin{array}{cc} 1& 0 \vspace{0.05in}\\
0 &-1 \end{array}\!\!\right),
\end{array}
\end{equation}
where $\psi(x,t; \lambda)=(\psi_1(x,t; \lambda),\psi_{2}(x,t; \lambda))^{\rm T}$ stands for the eigenvector, $\lambda\in\mathbb{C}$ is the spectral parameter. Moreover, Eq.~(\ref{dnls}) also possesses an infinite number of conservation laws, for example,
\begin{align}
\label{con-laws}
\left\{\begin{aligned}
&H_0=\int_{\mathbb{R}}|q|^2dx, \\
&H_{1}=\mathrm{Im}\int_{\mathbb{R}}q^*q_xdx+\frac12\int_{\mathbb{R}}|q|^4dx,\\
&H_{2}=\int_{\mathbb{R}}|q_x|^2-\frac32\mathrm{Im}(|q|^2qq_x^*)+\frac12|q|^6dx,
\end{aligned}\right.
\end{align}
where the star denotes the complex conjugate.

Note that the DNLS equation (\ref{dnls}) is $L^2$-norm being invariant under the scaling:
\bee
q(x,t)\rightarrow z^{\frac12}q(zx,z^2t),\quad z>0.
\ene

The inverse scattering transform (IST) was investigated for the DNLS equation (\ref{dnls}) with zero boundary conditions (ZBCs) to obtain its one-soliton solution~\cite{Kaup1978} and $N$-soliton solutions~\cite{Zhou2007}. The IST was also considered for the DNLS equation (\ref{dnls}) with non-zero boundary conditions (NZBCs)~\cite{Kawata1978, Chen2004, Chen2006, Lashkin2007}. The explicit double-pole solutions were found for the DNSL equation (\ref{dnls}) with ZBCs/NZBCs by the ISTs with the matrix Riemann-Hilbert problems~\cite{zhang-yan-jns20}. The long-time leading-order asymptotic behavior was analyzed for the DNLS equation (\ref{dnls})~\cite{RH1,RH3} via the Deift-Zhou's method~\cite{RH2}.

The local well-posedness for the Eq.~(\ref{dnls}) was proved in the energy space $H^1(\mathbb{R})$~\cite{Hayashi-1992,Hayashi-1993}. By using mass and energy conservation laws, Hayashi and Ozawa~\cite{Hayashi-1,Ozawa-1} proved that Eq.~(\ref{dnls}) was
global well-posedness in the energy space $H^1(\mathbb{R})$ under the following condition:
\bee\label{L2-finite}
||q(x,0)||_{L^2}<\sqrt{2\pi}.
\ene
Then, the condition (\ref{L2-finite}) was improved by Wu~\cite{wu-1,wu-2}. Moreover, Guo and Wu~\cite{wu-3} proved that Eq.~(\ref{dnls}) is
globally well-posed in the energy space $H^{\frac12}(\mathbb{R})$. Recently, the global existence of the DNLS equation (\ref{dnls}) was studied by the IST~\cite{ge,ge1,ge2,ge3}.  Moreover, Bahouri and Perelman~\cite{dnls-Bahouri} showed that the DNLS equation (\ref{dnls}) was globally well-posed for general Cauchy condition in $H^{1/2}(\mathbb{R})$ and that the $H^{1/2}$-norm of the solutions still remained globally bounded in time.

Recently, Koch and Tataru~\cite{dnls-Koch} studied the (de)focusing cubic nonlinear Schr\"odinger (NLS) equation
\bee \label{nls}
iq_{t}+q_{xx}\pm 2|q|^2q=0,\quad q=q(x,t),
\ene
and provided a modified conservation function for the NLS equation (\ref{nls}), and showed that there existed a conserved energy which is equivalent to the $H^s$-norm of the solution for each $s>1/2$ with the aid of IST.
Then, Koch and Liao~\cite{dnls-Koch-1} studied the one dimensional Gross-Pitaevskii (GP) equation
\bee\label{GP}
iq_{t}+q_{xx}=2q(|q|^2-1),\quad q=q(x,t),
\ene
and proved the global-in-time well-posedness of the GP equation (\ref{GP}) in the energy space. Recently, they~\cite{dnls-Koch-2} further constructed a family of conserved energies for the one dimensional Gross-Pitaevskii equation (\ref{GP}), but in the low regularity case.

In this paper, motivated by the idea for the NLS equation~\cite{dnls-Koch}, we prove the well-posedness results of scattering data for the DNLS equation (\ref{dnls}) with initial data $q(x)\in H^s(\mathbb{R})(s\geq\frac12)$ in the energy space, which is a complete metric space equipped with a newly introduced metric and the energy norm describing the $H^{s}(\mathbb{R})$ regularities of the solutions. We provide some regularity properties of transmission coefficient related to scattering data in $H^{s}(\mathbb{R})$.

 %we will prove that we can extend this countable family of conservation laws to a continuous family. For all $s>\frac12$ , we construct conserved energies $E^s$ associated to $H^s$ solutions for the DNLS equation (\ref{dnls}). Our construction of these energies relies heavily on the scattering transform associated to these problems, which requires some extensive preliminaries.

The main conclusion of this paper is the following theorem.

\begin{theorem}\label{1-th3}
Let $q(x)\in L^2(\mathbb{R})$ and $s_{11}(\lambda)$ be the reciprocal of the transmission coefficient of the modified Zakharov-Shabat spectral problem (\ref{lax}) associated to the DNLS equation (\ref{dnls}). Then one has the following properties:

(1) $\ln s_{11}(\lambda)=\sum_{j=1}^{\infty}b_{2j}(\lambda)$ with
\bee
b_{2j}(\lambda)=(-1)^{j}\int_{\Sigma_j}\lambda^{2j}\prod_{k=1}^{j}q(y_k)q^*(x_k)e^{2i\lambda^2(y_k-x_k)}dx_1dy_1\cdots dx_jdy_j,
\ene
being formal linear combinations of connected integrals, where $\Sigma_j$ is any domain which obeys the condition $x_k<y_k$ for all $k~(k\leq j)$, $\lambda\in\mathbb{C}$ is a spectral parameter, and the star denotes the complex conjugate.

(2) The following estimates hold:
\bee\label{coro1-1}
\ln s_{11}(\lambda)\sim -\frac{i}{2}||q(x)||^2_{L^2(\mathbb{R})},\quad \lambda\rightarrow\infty,
\ene
and
\bee\label{coro1-2}
s_{11}(\lambda)\sim e^{-\frac{i}{2}||q(x)||^2_{L^2(\mathbb{R})}},\quad \lambda\rightarrow\infty.
\ene
\end{theorem}

The rest of this paper is arranged as follows. In Sec. 2, we introduce some basic properties about the inverse scattering transform of the DNLS equation (\ref{dnls}) with $q(x)\in\mathcal{S}(\mathbb{R})$. In Sec. 3, we give the formal expansions of the reciprocal of the transmission coefficient, $s_{11}(\lambda)$, and its logarithmic function $\ln s_{11}(\lambda)$. In Sec. 4, we construct  iterative integrals $B_j(\lambda)$ arising from a formal expansion of  $\ln s_{11}(\lambda)$ into a Hopf algebra such that we can proof  the first conclusion of Theorem \ref{1-th3}. In Sec. 5, we give the boundary	estimate for the leading term in both $s_{11}(\lambda)-1$ and $\ln s_{11}(\lambda)$. In Sec. 6, we recall the function spaces $U^p,V^p$ and $DU^p$ and give the boundary	estimate for the iterative integrals $s_{2j}$ of $s_{11}(\lambda)-1$ with $q(x)\in H^s(\mathbb{R})$. In Sec. 7, we have the asymptotic expressions for $b_4(\lambda)$ and $b_6(\lambda)$. In Sec. 8, we give the expansions for the iterative integrals $b_{2j}(\lambda)$ with $q(x)\in H^s(\mathbb{R})$. Finally, we give some conclusions in Sec. 9.

\section{Preliminaries: Jost solutions and scattering data}

In this section, we review some basic properties about the inverse scattering transform of the DNLS equation (\ref{dnls}) with $q(x)\in\mathcal{S}(\mathbb{R})$ ($\mathcal{S}(\mathbb{R})$ represents Schwarz space)~\cite{ge,ge1,ge2,ge3,dnls-Bahouri,Ablowitz-2,Lee-1,Pelinovsky-1}. For the Lax pair (\ref{lax}) of the DNLS equation (\ref{dnls}), it is easy to see that the compatibility condition, $U_t-V_x+[U, V]=0$ (i.e., zero-curvature equation), of the Lax pair (\ref{lax}) just generates the DNLS equation (\ref{dnls}).% Namely, $q$ satisfies the DNLS equation (\ref{dnls}) if and only if
%\bee
%U_t-V_x+[U, V]=0.
%\ene
The zero-curvature equation has the advantage that it is well defined even without decay assumptions on the initial data, since it is all formal calculations.

For the given $q(x)\in\mathcal{S}(\mathbb{R})$, i.e., the potential $q(x)\rightarrow0$ as $x\rightarrow\pm\infty$, one has the asymptotics of Jost solutions (eigenfunctions) of  Lax pair (\ref{lax}) as
\bee\no
\psi_x=\left(\!\!\begin{array}{cc}
-i\lambda^2& 0  \vspace{0.05in}\\
0 & i\lambda^2
\end{array}\!\!\right)\psi,\quad x\rightarrow\pm\infty.
\ene

Therefore, it is natural to introduce the eigenfunction defined by the following boundary conditions
\begin{align}
\begin{aligned}
&\phi(x,\lambda)\sim e^{-i\lambda^2x}\begin{pmatrix} 1   \\ 0 \end{pmatrix},\quad x\rightarrow-\infty,\vspace{0.05in}\\
&\overline{\phi}(x,\lambda)\sim e^{i\lambda^2x}\begin{pmatrix} 0   \\ 1 \end{pmatrix},\quad x\rightarrow-\infty,\vspace{0.05in}\\
&\varphi(x,\lambda)\sim e^{i\lambda^2x}\begin{pmatrix} 0   \\ 1 \end{pmatrix},\quad x\rightarrow+\infty,\vspace{0.05in}\\
&\overline{\varphi}(x,\lambda)\sim e^{-i\lambda^2x}\begin{pmatrix} 1   \\ 0 \end{pmatrix},\quad x\rightarrow+\infty.
\end{aligned}
\end{align}

The functions $\phi(x,\lambda), \overline{\phi}(x,\lambda), \varphi(x,\lambda)$ and $\overline{\varphi}(x,\lambda)$ are called Jost solutions. The Jost solution $\phi(x,\lambda), \varphi(x,\lambda)$ can be analytically extended to $L_+=\{\lambda\in\mathbb{C}|\mathrm{Im}\lambda^2>0\}$, $C^\infty$ up to the boundary. The Jost solution $\overline{\phi}(x,\lambda), \overline{\varphi}(x,\lambda)$ can be analytically extended to $L_-=\{\lambda\in\mathbb{C}|\mathrm{Im}\lambda^2<0\}$, $C^\infty$ up to the boundary.

For $\lambda\in\mathbb{R}\cup i\mathbb{R}$, since the Jost solutions solve the both parts of the modified Zakharov-Shabat eigenvalue problem (\ref{lax}), there is a constant scattering matrix $S(\lambda)=(s_{ij})_{2\times 2}$ independent of $x,\,t$, which holds the following relation:
\bee\label{s11s12}
\left(\phi(x,\lambda),\overline{\phi}(x,\lambda)\right)=\left(\overline{\varphi}(x,\lambda),\varphi(x,\lambda)\right)\left(\!\!\begin{array}{cc}
s_{11}(\lambda)& s_{12}(\lambda)  \vspace{0.05in}\\
s_{21}(\lambda) & s_{22}(\lambda)
\end{array}\!\!\right)
\ene

The functions $s_{11}(\lambda)^{-1},s_{22}(\lambda)^{-1}$ are called transmission coefficients and $\frac{s_{21}(\lambda)}{s_{11}(\lambda)},\frac{s_{12}(\lambda)}{s_{22}(\lambda)}$ are called reflection coefficients. The Jost solutions have the following symmetry:
\begin{align}\label{jost-symmetry}
\begin{aligned}
&\phi(x,\lambda)=\sigma_3\phi(x,-\lambda),\quad \varphi(x,\lambda)=-\sigma_3\varphi(x,-\lambda),\vspace{0.05in}\\
&\phi(x,\lambda)=\left(\!\!\begin{array}{cc}
0& 1  \vspace{0.05in}\\
-1 & 0
\end{array}\!\!\right)\overline{\phi}^*(x,\lambda^*),\quad\varphi(x,\lambda)=\left(\!\!\begin{array}{cc}
0& -1  \vspace{0.05in}\\
1 & 0
\end{array}\!\!\right)\overline{\varphi}^*(x,\lambda^*).
\end{aligned}
\end{align}

According to Eqs.~(\ref{s11s12}) and (\ref{jost-symmetry}), the symmetry of the scattering data can be obtained. For $\lambda\in\mathbb{R}\cup i\mathbb{R}$,
\begin{align}\no
\begin{aligned}
&s_{11}(\lambda)=s_{11}(-\lambda),\quad s_{11}(\lambda)=s_{22}^*(\lambda^*),\vspace{0.05in}\\
&s_{21}(\lambda)=-s_{21}(-\lambda),\quad s_{12}(\lambda)=-s_{21}^*(\lambda^*).
\end{aligned}
\end{align}

Due to $\det(S(\lambda))=1$, there is the following equation:
\begin{align}\no
\begin{aligned}
&|s_{11}(\lambda)|^2+|s_{21}(\lambda)|^2=1,\quad \lambda\in\mathbb{R},\vspace{0.05in}\\
&|s_{11}(\lambda)|^2-|s_{21}(\lambda)|^2=1,\quad \lambda\in i\mathbb{R}.
\end{aligned}
\end{align}

It follows from Eq.~(\ref{s11s12}) that the scattering data have the Wronskian representations:
\begin{align}\label{det-1}
\begin{aligned}
&s_{11}(\lambda)=\det(\phi(x,\lambda),\varphi(x,\lambda)),\quad s_{12}(\lambda)=\det(\overline{\phi}(x,\lambda),\varphi(x,\lambda)),\vspace{0.05in}\\
&s_{21}(\lambda)=-\det(\phi(x,\lambda),\overline{\varphi}(x,\lambda)),\quad
s_{22}(\lambda)=-\det(\overline{\phi}(x,\lambda),\overline{\varphi}(x,\lambda)).
\end{aligned}
\end{align}

Denoting $\overline{s_{11}}(\lambda)=e^{\frac{i}{2}||q(x)||^2_{L^2(\mathbb{R})}}s_{11}(\sqrt{\lambda})$, it has the following asymptotic expansion:
\bee
\ln\overline{s_{11}}(\lambda)=\sum_{j=1}^{\infty}D_k(q)\lambda^{-j},\quad \lambda\rightarrow\infty,
\ene
where $D_k(q)$ are polynomial with respect to $q(x)$ and its derivatives. For example,
\bee
D_1(q)=\frac{i}{4}H_1,\quad D_2(q)=-\frac{i}{8}H_2.
\ene

Furthermore, one can show that $|\overline{s_{11}}(\lambda)|^2\in 1+\mathcal{S}(\mathbb{R})$, and
\bee\no
|\overline{s_{11}}(\lambda)|\geq1,~\lambda<0,\quad |\overline{s_{11}}(\lambda)|\leq1,~\lambda>0.
\ene

The scattering data satisfies the following time evolution equation:
\bee
\frac{\partial s_{11}(\lambda)}{\partial t}=0,\quad \frac{\partial s_{21}(\lambda)}{\partial t}=-4i\lambda^4s_{21}(\lambda)
\ene

Although the assumption $q(x)\in\mathcal{S}(\mathbb{R})$ can be weakened~\cite{ge,ge1,ge2,ge3,Pelinovsky-1}, one needs at least $q(x)\in L^1(\mathbb{R})$ to define the scattering data. A way to overcome this difficulty and to keep a trace of the complete integrability for $H^s$ solutions, for $\lambda\in L^+$, that remains well defined via Eqs.~(\ref{det-1}) for $q(x)\in L^2(\mathbb{R})$~\cite{dnls-Bahouri}.

\section{Global well-posedness}

In this section, we first consider the global well-posedness of the solutions for the Cauchy problems of the DNLS equation (\ref{dnls}):
\begin{equation}\label{dnls1}
\left\{ \begin{array}{l}
iq_{t}+q_{xx}\pm i(|q|^2q)_x=0,\vspace{0.1in}\\
q(x,0)=q_0(x)\in H^s(\mathbb{R}).
\end{array}\right.
\end{equation}

\begin{lemma}\label{le1}\cite{dnls-Bahouri} For any initial data $q_0(x)\in H^{\frac12}(\mathbb{R})$, the Cauchy problem  (\ref{dnls1}) is globally well-posed, and the corresponding solution $q(t)$ satisfies:
\bee
\mathrm{sup}_{t\in\mathbb{R}}||q(t)||_{H^{\frac12}(\mathbb{R})}<+\infty.
\ene
Moreover, if the initial datum is in $H^s(\mathbb{R})$ for some $s>\frac12$, then the $H^s$-norm of the solution of the Cauchy problem (\ref{dnls1}) remains globally bounded in time as well.
\end{lemma}

The scattering transform associated to the DNLS equation (\ref{dnls}) is defined via the first equation of (\ref{lax}), which can
be written as a linear system:
\bee \label{lax-x}
\left\{\begin{array}{l}
\dfrac{d\psi_1}{dx}=-i\lambda^2\psi_1+i\lambda q\psi_2, \v\\
\dfrac{d\psi_2}{dx}=i\lambda^2\psi_2+i\lambda q^*\psi_1,
\end{array}\right.
\ene

Then, based on the asymptotic of $q_0(x)$, one can seek for the Jost solutions $\psi_l$ with asymptotics:
\begin{align}\label{jost}
\left\{ \begin{aligned}
&\psi_{l}\sim\left(\!\!\begin{array}{cc} e^{-i\lambda^2x}\vspace{0.05in}\\
0\end{array}\!\!\right),\quad x\rightarrow -\infty\vspace{0.1in}\\
&\psi_{l}\sim\left(\!\!\begin{array}{cc} s_{11}(\lambda)e^{-i\lambda^2x}\vspace{0.05in}\\
s_{21}(\lambda)e^{i\lambda^2x}\end{array}\!\!\right),\quad x\rightarrow +\infty
\end{aligned}\right.
\end{align}
where $s_{11}^{-1}(\lambda)$ is the transmission coefficient and $\dfrac{s_{21}}{s_{11}}(\lambda)$ is the reflection coefficients.

\begin{theorem}\label{th1} The reciprocal of the transmission coefficient, $s_{11}(\lambda)$, has a formal expansion as follows:
\bee
s_{11}(\lambda)=1+\sum\limits_{j=1}^{\infty}s_{2j}(\lambda),
\ene
where $s_{2j}(\lambda)$'s are multi-linear integral forms with homogeneous of degree $2j$ in the potential $q$ and its conjugate $q^*$, that is,
\bee
s_{2j}(\lambda)=(-1)^{j}\int_{x_1<y_1<\cdots<x_j<y_j}\lambda^{2j}\prod_{k=1}^{j}q(y_k)q^*(x_k)e^{2i\lambda^2(y_k-x_k)}dx_1dy_1\cdots dx_jdy_j.
\ene
\end{theorem}

\begin{proof}
We solve system (\ref{lax-x}) by using the iterative method to prove this theorem. Firstly, we choose the initial value iteration function as:
\bee\label{value0}
\psi_{l}^{(0)}(x)=\left(\!\!\begin{array}{cc} e^{-i\lambda^2x}\vspace{0.05in}\\
0\end{array}\!\!\right),
\ene
where the upper right corner represents the number of iterations.

Substituting Eq.~(\ref{value0}) into the second one of Eq.~(\ref{lax-x}) yields
\bee\label{eq1}
\psi_{l2,x}^{(1)}=i\lambda^2\psi_{l2}^{(1)}-\lambda q^*e^{-i\lambda^2x}.
\ene

By solving ordinary differential equation (\ref{eq1}), we have
\bee\label{value1}
\psi_{l}^{(1)}(x)=\left(\!\!\begin{array}{cc} e^{-i\lambda^2x}\vspace{0.05in}\\
-e^{i\lambda^2x}\displaystyle\int_{-\infty}^{x}\lambda q^*(x_1)e^{-2i\lambda^2x_1}dx_1\end{array}\!\!\right),
\ene

Substituting the second component of Eq.~(\ref{value1}) into the first one of system (\ref{lax-x}) yields
\bee\label{eq2}
\psi_{l1,x}^{(2)}=-i\lambda^2\psi_{l1}^{(2)}+\lambda q\psi_{l2}^{(1)},
\ene
and solving ordinary differential equation (\ref{eq2}) has
\bee\label{value1-5}
\psi_{l1,x}^{(2)}=e^{-i\lambda^2x}-\int_{-\infty}^{x}\lambda q(y_1)e^{-i\lambda^2(x-y_1)}\int_{-\infty}^{y_1}\lambda q^*(x_1)e^{i\lambda^2(y_1-2x_1)}dx_1dy_1.
\ene

Simplifying Eq.~(\ref{value1-5}) and using the second component of Eq.~(\ref{value1}) yield
\bee\label{value2}
\psi_{l}^{(2)}(x)=\left(\!\!\begin{array}{cc} e^{-i\lambda^2x}\left(1-\d\int_{x_1<y_1<x}\lambda^2q(y_1)q^*(x_1)e^{2i\lambda^2(y_1-x_1)}dx_1dy_1\right)\vspace{0.1in}\\
-e^{i\lambda^2x}\d\int_{-\infty}^{x}\lambda q^*(x_1)e^{-2i\lambda^2x_1}dx_1\end{array}\!\!\right).
\ene

By repeating the above process, we can obtain the results of the third and fourth iterations as follows.
\bee\label{value3}
\psi_{l}^{(3)}(x)=\left(\!\!\begin{array}{cc} e^{-i\lambda^2x}\left(1-\d\int_{x_1<y_1<x}\lambda^2q(y_1)q^*(x_1)e^{2i\lambda^2(y_1-x_1)}dx_1dy_1\right)\vspace{0.1in}\\
-e^{i\lambda^2x}\d\int_{-\infty}^{x}\lambda q^*(x_2)e^{-2i\lambda^2x_2}\left(1-\int_{x_1<y_1<x_2}\lambda^2q(y_1)q^*(x_1)e^{2i\lambda^2(y_1-x_1)}dx_1dy_1\right)dx_2\end{array}\!\!\right),
\ene

and
\bee\label{value4}
\psi_{l}^{(4)}(x)=\left(\!\!\begin{array}{cc} e^{-i\lambda^2x}\left(1+\int_{-\infty}^{x}\lambda q(y_2)e^{i\lambda^2y_2}\psi_{l2}^{(3)}(y_2)dy_2\right)\vspace{0.1in}\\
-e^{i\lambda^2x}\d\int_{-\infty}^{x}\lambda q^*(x_2)e^{-2i\lambda^2x_2}\left(1-\int_{x_1<y_1<x_2}\lambda^2q(y_1)q^*(x_1)e^{2i\lambda^2(y_1-x_1)}dx_1dy_1\right)dx_2\end{array}\!\!\right).
\ene

Simplifying Eq.~(\ref{value4}) yields
\begin{align}\label{value4-5}
\begin{aligned}
\psi_{l1}^{(4)}(x)=&e^{-i\lambda^2x}\bigg(1-\d\int_{x_2<y_2<x}\lambda^2q(y_2)q^*(x_2)e^{2i\lambda^2(y_2-x_2)}\vspace{0.1in}\\
&\qquad +\d\int_{x_1<y_1<x_2<y_2<x}\lambda^4q(y_1)q^*(x_1)q(y_2)q^*(x_2)e^{2i\lambda^2(y_1+y_2-x_1-x_2)}dx_1dy_1dx_2dy_2\bigg).
\end{aligned}
\end{align}

According to Eqs.~(\ref{lax-x}) and (\ref{jost}), we iterate the above procedure and obtain the following expression
\begin{align}
\begin{aligned}
&s_2(\lambda)=-\int_{x_1<y_1}\lambda^2q(y_1)q^*(x_1)e^{2i\lambda^2(y_1-x_1)}dx_1dy_1,\vspace{0.05in}\\
&s_4(\lambda)=\int_{x_1<y_1<x_2<y_2}\lambda^4q(y_1)q^*(x_1)q(y_2)q^*(x_2)e^{2i\lambda^2(y_1+y_2-x_1-x_2)}dx_1dy_1dx_2dy_2,\vspace{0.05in}\\
&\qquad\qquad\qquad\qquad\qquad\qquad\qquad\vdots\vspace{0.05in}\\
&s_{2j}(\lambda)=(-1)^{j}\int_{x_1<y_1<\cdots<x_j<y_j}\lambda^{2j}\prod_{k=1}^{j}q(y_k)q^*(x_k)e^{2i\lambda^2(y_k-x_k)}dx_1dy_1\cdots dx_jdy_j.
\end{aligned}
\end{align}

Thus the proof is completed.
\end{proof}

We remark that, at least as long as $q(x)\in L^2(\mathbb{R})$, each term $s_{2j}(\lambda)$ is pointwise defined for$\lambda\in L^+$. For convenience, we need the formal series of $\ln s_{11}(\lambda)$ even more. So we will propose the following theorem.

\begin{theorem}\label{th2} The function $\ln s_{11}(\lambda)$ has a formal expansion as follows:
\bee
\ln s_{11}(\lambda)=\sum_{j=1}^{\infty}b_{2j}(\lambda),
\ene
where each component $b_{2j}(\lambda)$ is a linear combination of the following expressions:
\bee \label{bij}
b_{2j}(\lambda)=(-1)^{j}\int_{\Sigma_j}\lambda^{2j}\prod_{k=1}^{j}q(y_k)q^*(x_k)e^{2i\lambda^2(y_k-x_k)}dx_1dy_1\cdots dx_jdy_j,
\ene
where  $\Sigma_j$ is any possible domain which obeys the condition $x_k<y_k$ for all $k$.
\end{theorem}

\begin{proof}
Expanding the Taylor expansion of $\ln s_{11}(\lambda)$, we have
\begin{align}
\begin{aligned}
\ln s_{11}(\lambda)=\ln\left(1+\sum\limits_{j=1}^{\infty}s_{2j}(\lambda)\right)
=\sum_{k=1}^{\infty}(-1)^{k-1}k^{-1}\left(\sum\limits_{j=1}^{\infty}s_{2j}(\lambda)\right)^k.
\end{aligned}
\end{align}
Based on the sum of subscripts, we have
\begin{align}
\left\{\begin{aligned}
&b_2(\lambda)=s_2(\lambda),\vspace{0.05in}\\
&b_4(\lambda)=s_4(\lambda)-\frac{s_2(\lambda)^2}{2},\vspace{0.05in}\\
&b_6(\lambda)=s_6(\lambda)-s_2(\lambda)s_4(\lambda)+\frac{s_2(\lambda)^3}{3},\vspace{0.05in}\\
&b_8(\lambda)=s_8(\lambda)-s_2(\lambda)s_6(\lambda)-\frac{s_4(\lambda)^2}{2}+s_2(\lambda)^2s_4(\lambda)-\frac{s_2(\lambda)^4}{4},\vspace{0.05in}\\
&b_{10}(\lambda)=s_{10}(\lambda)-s_2(\lambda)s_8(\lambda)-s_4(\lambda)s_6(\lambda)+s_2(\lambda)^2s_6(\lambda)+s_2(\lambda)s_4(\lambda)^2
-s_2(\lambda)^3s_4(\lambda)+\frac{s_2(\lambda)^5}{5},\vspace{0.05in}\\
&\qquad\qquad\qquad\qquad\qquad\qquad\cdots\vspace{0.05in}\\
\end{aligned}\right.
\end{align}
Thus the proof is completed.
\end{proof}

\section{Hopf algebra}

The goal of this section is to construct iterative integrals given by Eq.~(\ref{bij}) into a Hopf algebra. Our iterative integrals are in the following form:
\bee\label{integral}
\int_{\Sigma_j}\prod_{k=1}^{j}q(y_k)q^*(x_k)e^{2i\lambda^2(y_k-x_k)}dx_1dy_1\cdots dx_jdy_j,
\ene
where $\Sigma_j$ is an appropriate domain which obeys $x_k<y_k$ for all $k$.

Omitting the indices, we use $X$ to represent $x_j$ and use $Y$ to represent $y_j$. For example,
\begin{align}\no
\begin{aligned}
&x_1<y_1\longrightarrow XY,\vspace{0.05in}\\
&x_1<y_1<x_2<y_2\longrightarrow XYXY,\vspace{0.05in}\\
&x_1<x_2<x_3<y_1<y_2<y_3\longrightarrow XXXYYY.
\end{aligned}
\end{align}

In what follows, we only consider the situation where the quantities of $X$ and $Y$ are the same and $x_l, y_l$ satisfy constraint conditions: $x_l<y_l$. Then, we can use letters $X, Y$ to simply represent iterative integrals. For instance,
\begin{align}\label{represent}
\begin{aligned}
&XY:=\int_{x_1<y_1}q(y_1)q^*(x_1)e^{2i\lambda^2(y_1-x_1)}dx_1dy_1,\vspace{0.05in}\\
&XYXY:=\int_{x_1<y_1<x_2<y_2}q(y_1)q^*(x_1)q(y_2)q^*(x_2)e^{2i\lambda^2(y_1+y_2-x_1-x_2)}dx_1dy_1dx_2dy_2,\vspace{0.05in}\\
&(XY)^{(j)}:=\int_{x_1<y_1<\cdots<x_j<y_j}\prod_{k=1}^{j}q(y_k)q^*(x_k)e^{2i\lambda^2(y_k-x_k)}dx_1dy_1\cdots dx_jdy_j,\vspace{0.05in}\\
&XXYXYY:=\int_{x_1<x_2<y_1<x_3<y_2<y_3}\prod_{j=1}^3q(y_j)q^*(x_j)e^{2i\lambda^2(y_1+y_2+y_3-x_1-x_2-x_3)}dx_1dy_1dx_2dy_2dx_3dy_3,\vspace{0.05in}\\
&XXXYYY:=\int_{x_1<x_2<x_3<y_1<y_2<y_3}\prod_{j=1}^3q(y_j)q^*(x_j)e^{2i\lambda^2(y_1+y_2+y_3-x_1-x_2-x_3)}dx_1dy_1dx_2dy_2dx_3dy_3.
\end{aligned}
\end{align}

According to Eq.~(\ref{represent}), $s_{11}(\lambda)$ can be expressed in the following form:
\bee
s_{11}(\lambda)=1+\sum\limits_{j=1}^{\infty}(-1)^j\lambda^{2j}(XY)^{(j)}.
\ene

And $b_{2j}(\lambda)$ in $\ln s_{11}(\lambda)$ can be expressed in the following form:
\begin{align}\label{XY-b}
\left\{\begin{aligned}
&b_2(\lambda)=-\lambda^2XY,\vspace{0.05in}\\
&b_4(\lambda)=\lambda^4(XY)^{(2)}-\frac{\lambda^4(XY)^2}{2},\vspace{0.05in}\\
&b_6(\lambda)=-\lambda^6(XY)^{(3)}+\lambda^6XY\times(XY)^{(2)}-\lambda^6\frac{(XY)^3}{3},\vspace{0.05in}\\
&b_8(\lambda)=\lambda^8(XY)^{(4)}-\lambda^8XY\times(XY)^{(3)}-\frac{\lambda^8((XY)^{(2)})^2}{2}+\lambda^8(XY)^2\times(XY)^{(2)}-\lambda^8\frac{(XY)^4}{4},\vspace{0.05in}\\
&\qquad\qquad\qquad\qquad\qquad\qquad\cdots\vspace{0.05in}\\
\end{aligned}\right.
\end{align}

Next we provide the pairing principle for $X$ and $Y$. Starting from left to right, each $X$ pairs with its nearest $Y$ (see Fig.~\ref{fig1}).

\begin{figure}[!t]
    \centering
\vspace{-0.05in}
  {\scalebox{0.35}[0.35]{\includegraphics{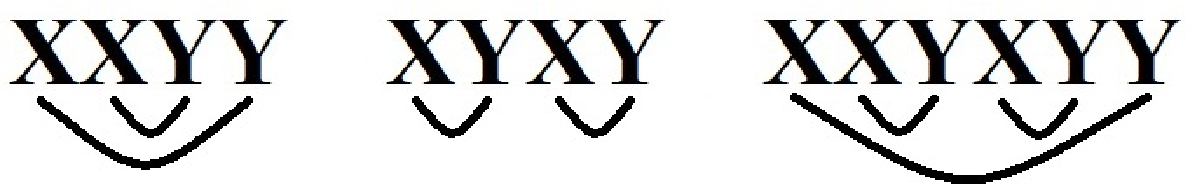}}}\hspace{-0.35in}
%\vspace{-0.03in}
\caption{The pairing principle for $X$ and $Y$.}
   \label{fig1}
\end{figure}

We call an integral (\ref{integral}) connected if its first $X$ is paired with its last $Y$. Now, we want to prove that $\ln s_{11}(\lambda)$ is composed of connected integrals. According to Fubini's theorem, we have
\begin{align}\no
\begin{aligned}
\left(\int_{x_1<y_1}f(x_1)g(y_1)dx_1dy_1\right)^2=&\,\, 2\int_{x_1<y_1<x_2<y_2}f(x_1)g(y_1)f(x_2)g(y_2)dx_1dy_1dx_2dy_2\vspace{0.1in}\\
&+4\int_{x_1<x_2<y_1<y_2}f(x_1)g(y_1)f(x_2)g(y_2)dx_1dy_1dx_2dy_2.
\end{aligned}
\end{align}

Represented by letters $X, Y$, the above equation is written as:
\bee\no
(XY)^2=2XYXY+4XXYY.
\ene

\begin{theorem}\label{th3}
$\ln s_{11}(\lambda)$ has a formal expansion as follows:
\bee
\ln s_{11}(\lambda)=-\lambda^2XY-2\lambda^4XXYY-4\lambda^6(XXYXYY+3XXXYYY)+\cdots,
\ene
\end{theorem}

\begin{proof}
According to Eq.~(\ref{XY-b}), we will calculate each item separately. Firstly, we have
\begin{align}
\begin{aligned}
b_4(\lambda)&=\lambda^4(XY)^{(2)}-\frac{\lambda^4(XY)^2}{2}\vspace{0.05in}\\
&=\lambda^4\left((XY)^{(2)}-\frac{2XYXY+4XXYY}{2}\right)\vspace{0.05in}\\
&=-2\lambda^4XXYY.
\end{aligned}
\end{align}

Expanding $b_6(\lambda)$ yields
\begin{align}\label{th3-1}
\begin{aligned}
b_6(\lambda)&=-\lambda^6(XY)^{(3)}+\lambda^6XY\times(XY)^{(2)}-\lambda^6\frac{(XY)^3}{3}\vspace{0.05in}\\
&=-\lambda^6\left((XY)^{(3)}-XY\times(XY)^{(2)}+\frac{(2XYXY+4XXYY)\times XY}{3}\right)\vspace{0.05in}\\
&=-\lambda^6\left((XY)^{(3)}-\frac{XY\times(XY)^{(2)}}{3}+\frac{4XXYY\times XY}{3}\right).
\end{aligned}
\end{align}

Since
\begin{align}\label{th3-2}
\begin{aligned}
XY\times(XY)^{(2)}=&X\times\left(\dot{Y}XYXY+X\dot{Y}YXY+XY\dot{Y}XY+XYX\dot{Y}Y+XYXY\dot{Y}\right)\vspace{0.05in}\\
=&XYXYXY+2XXYYXY+2XXYYXY+XYXYXY+2XXYXYY\vspace{0.05in}\\
&+2XYXXYY+2XXYXYY+2XYXXYY+XYXYXY\vspace{0.05in}\\
=&3(XY)^{(3)}+4XXYYXY+4XXYXYY+4XYXXYY,
\end{aligned}
\end{align}
and
\begin{align}\label{th3-3}
\begin{aligned}
XXYY\times XY=&\left(XXYY\dot{X}+XXY\dot{X}Y+XX\dot{X}YY+X\dot{X}XYY+\dot{X}XXYY\right)\times Y\vspace{0.05in}\\
=&XXYYXY+2XXYXYY+3XXXYYY+3XXXYYY+XXYXYY\vspace{0.05in}\\
&+3XXXYYY+XXYXYY+XYXXYY\vspace{0.05in}\\
=&XXYYXY+4XXYXYY+9XXXYYY+XYXXYY.
\end{aligned}
\end{align}

According to Eqs.~(\ref{th3-1})-(\ref{th3-3}), we have
\bee\no
b_6(\lambda)=-4\lambda^6(XXYXYY+3XXXYYY).
\ene

Thus the proof is completed.
\end{proof}

We need to construct a Hopf algebra. Let $H$ be a graded algebra, and there are the following operations on $H$:
\begin{align}\no
\begin{aligned}
& \times: natural~multiplication,\quad  \otimes: tensor~product,\quad
 \Delta: H\rightarrow H\times H,
\end{aligned}
\end{align}
where
\bee\no
\Delta a=\sum\limits_{a_1a_2=a}a_1\otimes a_2.
\ene

And we call a word $a\in H$ group-like if
\bee\no
\Delta a=a\otimes a.
\ene

Then, we know that the set $G$ of all group-like words with natural multiplication is a group. The primitive words are
\bee\no
P=\{p\in H| \Delta p=1\otimes p+p\otimes 1\}.
\ene

We note that the primitive words are linear combinations of connected integrals. There is a relationship between group $G$ and $P$ as follows:
\bee
G=e^{P}.
\ene

And we get the following lemma.
\begin{lemma}\label{le2}
The expression
\bee
s_{11}(\lambda)=1+\sum\limits_{j=1}^{\infty}(-1)^j\lambda^{2j}(XY)^{(j)}.
\ene
belongs to $G$.
\end{lemma}

\begin{proof}
For the sake of simplicity, let $(XY)^{(0)}=1$. Then
\begin{align}\no
\begin{aligned}
\Delta s_{11}(\lambda)&=\Delta\sum\limits_{j=0}^{\infty}(-1)^j\lambda^{2j}(XY)^{(j)}\vspace{0.05in}\\
&=\sum\limits_{j=0}^{\infty}\Delta(-1)^j\lambda^{2j}(XY)^{(j)}\vspace{0.05in}\\
&=\sum\limits_{j=0}^{\infty}\sum\limits_{k=0}^{n}(-1)^k\lambda^{2k}(XY)^{(k)}\otimes(-1)^{j-k}\lambda^{2(j-k)}(XY)^{(j-k)}\vspace{0.05in}\\
&=s_{11}(\lambda)\otimes s_{11}(\lambda).
\end{aligned}
\end{align}

Thus the proof is completed.
\end{proof}

So $b_{2j}(\lambda)'s$ are formal linear combinations of connected integrals. Then, the proof of the first conclusion of Theorem \ref{1-th3} is completed.

\section{Bounding the integral term $s_{2}(\lambda)$}

The leading term in both $s_{11}(\lambda)-1$ and $\ln s_{11}(\lambda)$ away from $\Sigma:=\{\lambda|\lambda^2\in\mathbb{R}\}$ is $s_2(\lambda)$. Thus we here analyze the term $s_{2}(\lambda)$.

Firstly, we know
\bee
s_2(\lambda)=-\int_{x<y}\lambda^2q(y)q^*(x)e^{2i\lambda^2(y-x)}dxdy.
\ene

For convenience, we choose the unitary Fourier transform
\bee\label{Fourier1}
\hat{q}(\xi)=\frac{1}{\sqrt{2\pi}}\int_{\mathbb{R}}q(x)e^{-ix\xi}dx,
\ene

and the corresponding Fourier inversion formula as follows:
\bee\label{Fourier2}
\check{q}(x)=\frac{1}{\sqrt{2\pi}}\int_{\mathbb{R}}q(\xi)e^{ix\xi}dx.
\ene

Let's note that $\delta(x)$ is a Dirac delta function and has properties:
\bee\label{Fourier3}
\int_{-\infty}^{+\infty}e^{i\omega(\xi-\eta)}d\omega=2\pi\delta(\xi-\eta).
\ene

According to Eq.~(\ref{Fourier1}), (\ref{Fourier2}) and (\ref{Fourier3}), we have
\begin{align}\label{Fourier4}
\begin{aligned}
s_2(\lambda)=&-\int_{x<y}\lambda^2q(y)q^*(x)e^{2i\lambda^2(y-x)}dxdy\vspace{0.05in}\\
=&-\frac{1}{2\pi}\int_{x<y}\int_{-\infty}^{+\infty}\int_{-\infty}^{+\infty}\lambda^2e^{2i\lambda^2(y-x)}\hat{q}(\xi)e^{iy\xi}\hat{q}^*(\eta)e^{-ix\eta}d\xi d\eta dxdy\vspace{0.05in}\\
=&-\frac{1}{2\pi}\int_{-\infty}^{+\infty}\int_{-\infty}^{+\infty}\int_{-\infty}^{+\infty}\lambda^2\left\{\int_{-\infty}^{y}e^{iy(2\lambda^2+\xi)-ix(2\lambda^2+\eta)}dx\right\}
\hat{q}(\xi)\hat{q}^*(\eta)d\xi d\eta dy\vspace{0.05in}\\
=&~\frac{1}{2\pi}\int_{-\infty}^{+\infty}\int_{-\infty}^{+\infty}\int_{-\infty}^{+\infty}\frac{\lambda^2}{2i\lambda^2+i\eta}e^{iy(\xi-\eta)}
\hat{q}(\xi)\hat{q}^*(\eta)d\xi d\eta dy\vspace{0.05in}\\
=&-\frac{i}{2\pi}\int_{-\infty}^{+\infty}\int_{-\infty}^{+\infty}\frac{\lambda^2}{2\lambda^2+\eta}\left\{\int_{-\infty}^{+\infty}e^{iy(\xi-\eta)}dy\right\}
\hat{q}(\xi)\hat{q}^*(\eta)d\xi d\eta \vspace{0.05in}\\
=&-i\int_{-\infty}^{+\infty}\int_{-\infty}^{+\infty}\frac{\lambda^2}{2\lambda^2+\eta}\delta(\xi-\eta)\hat{q}(\xi)\hat{q}^*(\eta)d\xi d\eta\vspace{0.05in}\\
=&-i\int_{-\infty}^{+\infty}\frac{\lambda^2}{2\lambda^2+\xi}|\hat{q}(\xi)|^2d\xi.
\end{aligned}
\end{align}

Simplifying the last equation of Eq.~(\ref{Fourier4}), we obtain
\bee
s_2(\lambda)=-\frac{i}{2}||q(x)||^2_{L^2}+i\int_{-\infty}^{+\infty}\frac{\xi}{4\lambda^2+2\xi}|\hat{q}(\xi)|^2d\xi.
\ene

For convenience, we introduce a new variable as follows:
\begin{align}
\left\{\begin{aligned}
&c_2(\lambda):=s_2(\lambda)+\frac{i}{2}||q(x)||^2_{L^2}=i\int_{-\infty}^{+\infty}\frac{\xi}{4\lambda^2+2\xi}|\hat{q}(\xi)|^2d\xi,\vspace{0.05in}\\
&c_{2j}(\lambda):=b_{2j}(\lambda),\quad j\geq2.
\end{aligned}\right.
\end{align}

Obviously, we have
\bee\no
\ln s_{11}(\lambda)+\frac{i}{2}||q(x)||^2_{L^2}=\sum_{j=1}^{\infty}c_{2j}(\lambda).
\ene

Through the above analysis, we obtain the following theorem.

\begin{theorem}\label{th5}
(a) For $\lambda\in\Sigma_+:=\{\lambda|{\rm Im}(\lambda^2)>0\}$, we have
\bee
|\mathrm{Re}~c_2(\lambda)|\leq\int_{-\infty}^{+\infty}\dfrac{|\xi|~\mathrm{Im}(\lambda^2)}
{(2\mathrm{Re}\lambda^2+\xi)^2+(2\mathrm{Im}(\lambda^2))^2}
|\hat{q}(\xi)|^2d\xi.
\ene

(b) For all $N\in\mathbb{N}$, we have
\bee
|\mathrm{Re}~(c_2(\frac{\lambda}{\sqrt{2}})-\frac{i}{2}\sum\limits_{k=0}^{N-1}M_{k,2}\lambda^{-2k-2})|\leq\frac{|\lambda|^{-2N}}{2}\int_{-\infty}^{+\infty}|\xi|^{N+1}
\dfrac{\mathrm{Im}\lambda^2+|\mathrm{Re}\lambda^2+\xi|}{(\mathrm{Re}\lambda^2+\xi)^2+(\mathrm{Im}\lambda^2)^2}|\hat{q}(\xi)|^2d\xi.
\ene
where
\begin{align}\no
\begin{aligned}
M_{k,2}&=\int_{-\infty}^{+\infty}(-1)^k\xi^{k+1}|\hat{q}(\xi)|^2d\xi\vspace{0.05in}\\
&=\begin{cases}
-\int_{-\infty}^{+\infty}|q^{(k)}|^2dx,\quad j=2k-1,\vspace{0.03in}\\
-\mathrm{Im}~\int_{-\infty}^{+\infty}q^{(k+1)}q^{*(k)}dx,\quad j=2k.
\end{cases}
\end{aligned}
\end{align}
\end{theorem}

\begin{proof}
Firstly, we prove the property (a). We have
\begin{align}\no
\begin{aligned}
|\mathrm{Re}~c_2(\lambda)|&=|\mathrm{Re}~i\int_{-\infty}^{+\infty}\frac{\xi}{4\lambda^2+2\xi}|\hat{q}(\xi)|^2d\xi|,\vspace{0.05in}\\
&=|\mathrm{Im}~\int_{-\infty}^{+\infty}\frac{\xi}{4\lambda^2+2\xi}|\hat{q}(\xi)|^2d\xi|,\vspace{0.05in}\\
&=|\int_{-\infty}^{+\infty}\frac{4\xi\mathrm{Im}\lambda^2}{(4\mathrm{Re}\lambda^2+2\xi)^2+(4\mathrm{Im}\lambda^2)^2}|\hat{q}(\xi)|^2d\xi|,\vspace{0.05in}\\
&\leq\int_{-\infty}^{+\infty}\dfrac{|\xi|~\mathrm{Im}\lambda^2}{(2\mathrm{Re}\lambda^2+\xi)^2+(2\mathrm{Im}\lambda^2)^2}
|\hat{q}(\xi)|^2d\xi.
\end{aligned}
\end{align}

Then we prove the property (b). For $c_2(\lambda)$, we can rewrite it as
\begin{align}\no
\begin{aligned}
c_2(\lambda)&=i\int_{-\infty}^{+\infty}\frac{\xi}{4\lambda^2+2\xi}|\hat{q}(\xi)|^2d\xi\vspace{0.05in}\\
&=\frac{i}{4\lambda^2}\int_{-\infty}^{+\infty}\frac{\xi}{1-(-\frac{\xi}{2\lambda^2})}|\hat{q}(\xi)|^2d\xi\vspace{0.05in}\\
&=\frac{i}{4\lambda^2}\sum\limits_{j=0}^{\infty}\int_{-\infty}^{+\infty}\xi(-\frac{\xi}{2\lambda^2})^j|\hat{q}(\xi)|^2d\xi\vspace{0.05in}\\
&=\frac{i}{2}\left(\sum\limits_{j=0}^{N-1}\int_{-\infty}^{+\infty}(-1)^j\xi^{j+1}(2\lambda^2)^{-j-1}|\hat{q}(\xi)|^2d\xi+
\sum\limits_{j=N}^{\infty}\int_{-\infty}^{+\infty}(-1)^j\xi^{j+1}(2\lambda^2)^{-j-1}|\hat{q}(\xi)|^2d\xi\right)\vspace{0.05in}\\
&=\frac{i}{2}\left(\sum\limits_{j=0}^{N-1}\int_{-\infty}^{+\infty}(-1)^j\xi^{j+1}(2\lambda^2)^{-j-1}|\hat{q}(\xi)|^2d\xi+
\int_{-\infty}^{+\infty}\frac{(-2\lambda^2)^{-N}\xi^{N+1}}{2\lambda^2+\xi}|\hat{q}(\xi)|^2d\xi\right).
\end{aligned}
\end{align}

Then we have
\begin{align}\no
\begin{aligned}
&|\mathrm{Re}~(c_2(\frac{\lambda}{\sqrt{2}})-\frac{i}{2}\sum\limits_{k=0}^{N-1}M_{k,2}\lambda^{-2k-2})|\vspace{0.05in}\\
=&~|\mathrm{Im}~\frac12\int_{-\infty}^{+\infty}\frac{(-\lambda^2)^{-N}\xi^{N+1}}{\lambda^2+\xi}|\hat{q}(\xi)|^2d\xi|\vspace{0.05in}\\
\leq&~\frac{|\lambda|^{-2N}}{2}\int_{-\infty}^{+\infty}|\xi|^{N+1}\dfrac{\mathrm{Im}\lambda^2+|\mathrm{Re}\lambda^2+\xi|}{(\mathrm{Re}\lambda^2+\xi)^2+
(\mathrm{Im}\lambda^2)^2}|\hat{q}(\xi)|^2d\xi.
\end{aligned}
\end{align}

Thus the proof is completed.
\end{proof}

%\section{Function spaces}

\section{Bounding the iterative integrals $s_{2j}$}

Here we first recall the function spaces $U^p,V^p$ and $DU^p$~\cite{dnls-Koch}.

\begin{defi}
(a) We define the space $V^p$ as the space of those function that the following norm is finite:
\bee\no
||v||_{V^p}=\sup\limits_{-\infty<t_1<t_2\cdots<t_N=+\infty}(\sum\limits_{j=1}^{N-1}|v(t_{j+1})-v(t_j)|^p)^{\frac{1}{p}},\quad 1<p<\infty,
\ene
where $v(t_N)=0$.

(b) A $U^p$ atom is defined as
\bee\no
u(x)=\sum\limits_{j=1}^{N-1}c_j\chi_{[t_j,t_{j+1})}(x),\quad if \sum\limits_{j=1}^{N-1}|c_j|^p\leq1,
\ene
where $\chi$ is the characteristic function as follows:
\bee
\chi_{[t_j,t_{j+1})}(t)=\begin{cases}
1,\quad t_j\leq t<t_{j+1}\vspace{0.03in}\\
0,\quad otherwise.
\end{cases}
\ene

We define the space $U^p$ as:
\bee\no
U^p=\left\{\sum\limits_{j=1}^{\infty}c_ja_j|(c_j)_j\in l^1,a_j~is~U^p~atom\right\},
\ene
with the following norm,
\bee\no
||u||_{U^p}=\inf\left\{\sum\limits_{j=1}^{\infty}|c_j|\big|u=\sum\limits_{j=1}^{\infty}c_ja_j,(c_j)_j\in l^1,a_j~is~U^p~atom\right\}.
\ene

(c) We define the space $DU^p$ as:
\bee\no
DU^p=\{u^{'}|u\in U^p\},
\ene
with the following norm,
\bee\no
||f||_{DU^p}=\sup\left\{\int_{-\infty}^{+\infty}f\phi dt\big|~||\phi||_{V^q}\leq1,~\phi\in C_c^{\infty}\right\}.
\ene

(d) We define the space $DV^p$ as:
\bee\no
DV^p=\{v^{'}|v\in V^p,v~is~left-continuous~functions~with~limit~0~at~the~right~endpoint~\},
\ene
with the following norm,
\bee\no
||f||_{DV^p}=\sup\left\{\int_{-\infty}^{+\infty}f\phi dt\big|~||\phi||_{U^q}\leq1,~\phi\in C_c^{\infty}\right\}.
\ene

(e) Let $\sigma>0$, we define
\bee\no
||u||_{l_{\sigma}^pU^2}=\big|\big|||\chi_{[\frac{k}{\sigma},\frac{k+1}{\sigma}]}u||_{U^2}\big|\big|_{l_k^p},
\ene
and
\bee\no
||u||_{l_{\sigma}^pDU^2}=\big|\big|||\chi_{[\frac{k}{\sigma},\frac{k+1}{\sigma}]}u||_{DU^2}\big|\big|_{l_k^p},
\ene
where $\chi_{[\frac{k}{\sigma},\frac{k+1}{\sigma}]}$ is a smooth cutoff function in interval $[\frac{k}{\sigma},\frac{k+1}{\sigma}]$.
\end{defi}

Let's recall some basic properties of the spaces $U^p,V^p$ and $DU^p$.
\begin{lemma}\label{le4}
(a) For all $1<p<\infty$, we have
\bee
U^p\subset V^p,~and~||u||_{V^p}\leq ||u||_{U^p}.
\ene

If $g\in L^1$, we have
\bee
||g*v||_{V^p}\leq||g||_{L^1}||v||_{V^p},\quad ||g*u||_{U^p}\leq||g||_{L^1}||u||_{U^p}.
\ene

(b) If $u\in U^2, v\in V^2$ and $v$ is left-continuous functions with limit 0 at the right endpoint, then
\bee
||u||_{U^2}=||u^{'}||_{DU^2},\quad ||v||_{V^2}=||v^{'}||_{DV^2}.
\ene

(c) The bilinear estimates
\bee
||vu||_{DU^2}\leq2||v||_{V^2}||u||_{DU^2}.
\ene
\end{lemma}

For convenience, we define the one-step operator as follows:
\bee\no
L(f)(t)=-\int_{x<y<t}q(y)q^*(x)e^{2i\lambda^2(y-x)}f(x)dxdy.
\ene

\begin{lemma}\label{le5}
For $\mathrm{Im}\lambda^2>0$, we have
\bee
||L||_{V^2\rightarrow U^2}\leq4\sqrt{2}||e^{-i\mathrm{Re}\lambda^2x}q||_{DU^2}^2.
\ene
\end{lemma}

\begin{proof}
It suffices to consider $\lambda^2=i$. Then, according to Lemma.~\ref{le4}, we have
\begin{align}\no
\begin{aligned}
||Lf||_{U^2}&=||(\int_{-\infty}^t\int_{-\infty}^{y}q(y)q^*(x)e^{2(x-y)}f(x)dxdy)^{'}||_{DU^2}\vspace{0.05in}\\
&=||\int_{-\infty}^{t}q(t)q^*(x)e^{2(x-t)}f(x)dx||_{DU^2}\vspace{0.05in}\\
&\leq2||q||_{DU^2}||\chi_{t<0}e^{2t}*(q^*f)||_{V^2}\vspace{0.05in}\\
&\leq2\sqrt{2}||q||_{DU^2}||\chi_{t<0}e^{2t}*(q^*f)||_{U^2}\vspace{0.05in}\\
&\leq2\sqrt{2}||q||_{DU^2}||(\chi_{t<0}e^{2t}*(q^*f))^{'}||_{DU^2}\vspace{0.05in}\\
&\leq4\sqrt{2}||q||_{DU^2}||\chi_{t<0}e^{2t}*(q^*f)||_{DU^2}\vspace{0.05in}\\
&\leq2\sqrt{2}||q||_{DU^2}||q^*f||_{DU^2}\vspace{0.05in}\\
&\leq4\sqrt{2}||q||_{DU^2}^2||f||_{V^2}.
\end{aligned}
\end{align}

Thus the proof is completed.
\end{proof}

This bound is very sharp on the region $\Sigma$, but we want to move $\lambda$ into the region $\Omega_+$. So we need the following Lemma.
\begin{lemma}\label{le6}
(a) We have
\bee
||q||_{l_\sigma^pU^2}\lesssim||\partial q||_{l_\sigma^pDU^2}+\sigma||q||_{l_\sigma^pDU^2}.
\ene

(b) The space $l_\sigma^2U^2$ can be seen as
\bee
l_\sigma^2U^2=DU^2+\sqrt{\sigma}L^2.
\ene

(c) The following relationship hold:
\bee
B_{2,1}^{-\frac12}\subset l_1^2U^2\subset B_{2,\infty}^{-\frac12}.
\ene

(d) For all $p>2$, we have
\bee
||q||_{l_\tau^pDU^2}\lesssim\tau^{\frac{1}{p}-1}||q||_{\dot{H}^{\frac12-\frac{1}{p}}}.
\ene

If $0\leq\tau_1\leq\tau_2$, then
\bee
||q||_{l_{\tau_2}^pDU^2}\lesssim||q||_{l_{\tau_1}^pDU^2}\lesssim(\frac{\tau_2}{\tau_1})^{1-\frac{1}{p}}||q||_{l_{\tau_2}^pDU^2}.
\ene

\end{lemma}

\begin{lemma}\label{le7}
For $\mathrm{Im}\lambda^2>0$, we have
\bee
||L||_{U^2\rightarrow U^2}\lesssim||e^{-i\mathrm{Re}\lambda^2x}q||_{l^2_{\mathrm{Im}\lambda^2}DU^2}^2.
\ene

\end{lemma}

\begin{proof}
It suffices to consider $\lambda^2=i$. Then, we have
\begin{align}\no
\begin{aligned}
||Lf||_{U^2}&=||\int_{-\infty}^{t}q(t)q^*(x)e^{2(x-t)}f(x)dx||_{DU^2}\vspace{0.05in}\\
&\lesssim||q||_{l^2DU^2}||\chi_{t<0}e^{2t}*(q^*f)||_{l^2U^2}\vspace{0.05in}\\
&\lesssim||q||_{l^2DU^2}||\left(\chi_{t<0}e^{2t}*(q^*f)\right)^{'}||_{l^2DU^2}\vspace{0.05in}\\
&\lesssim||q||_{l^2DU^2}||q^*f||_{l^2DU^2}\vspace{0.05in}\\
&\lesssim||q||_{l^2DU^2}^2||f||_{U^2}.
\end{aligned}
\end{align}

Thus the proof is completed.
\end{proof}

Based on the above analysis, we will provide an estimate of $s_{2j}(\lambda)$ and $b_{2j}(\lambda)$.

\begin{theorem}\label{th6}
The iterated integrals $s_{2j}(\lambda)$ and $b_{2j}(\lambda)$ have the following estimate:
\bee
|\lambda^{-2j}s_{2j}(\lambda)|+|\lambda^{-2j}b_{2j}(\lambda)|\leq C||e^{-i\mathrm{Re}\lambda^2x}q||_{l^2_{\mathrm{Im}\lambda^2}DU^2}^{2j}.
\ene

\end{theorem}

\begin{proof}
According to Theorem \ref{th1}, the first component of the Jost solution can be rewritten as
\bee\no
\psi_1(x)=e^{-i\lambda^2x}\sum\limits_{j=0}^{\infty}\lambda^{2j}L^j1(x).
\ene

Then, the transmission coefficient $s_{11}(\lambda)$ can be expressed as
\bee\no
s_{11}(\lambda)=\lim\limits_{x\rightarrow +\infty}\sum\limits_{j=0}^{\infty}\lambda^{2j}L^j1(x).
\ene

Firstly, we will introduce a partial order $\preceq$. $f_1\preceq f_2$ means that each coefficient of Taylor expansion at zero for $f_1$ is not greater than the coefficient of Taylor expansion at zero for $f_2$. Because of both the iterated integrals $s_{2j}$ and $b_{2j}$ are homogeneous forms, we have
\bee\no
\sum\limits_{j=1}^{\infty}(z\lambda^{-2})^jb_{2j}=\ln\left(1+\sum\limits_{j=1}^{\infty}(z\lambda^{-2})^js_{2j}\right).
\ene

Let $f: z\rightarrow \frac{z}{1-z}$, and we note that $\ln(1+z)\preceq f(z)$. Then we have
\bee\label{th6-1}
\sum\limits_{j=1}^{\infty}(z\lambda^{-2})^jb_{2j}\preceq f(f(C_1z)),
\ene
where
\bee\no
C_1=(\frac{C}{2^{j-1}})^{\frac{1}{j}}||e^{-i\mathrm{Re}\lambda^2x}q||_{l^2_{\mathrm{Im}\lambda^2}DU^2}^{2}
\ene

Simplifying Eq.~(\ref{th6-1}) yields
\bee
\sum\limits_{j=1}^{\infty}(z\lambda^{-2})^jb_{2j}\preceq \sum\limits_{j=0}^{\infty}2^j(C_1z)^{j+1}.
\ene

Comparing the coefficients of each power of $z$, we get
\bee\no
|\lambda^{-2j}b_{2j}(\lambda)|\leq 2^{j-1}C_1^j.
\ene

Thus the proof is completed.
\end{proof}

\begin{theorem}\label{th7}
Suppose that $q\in H^s$. If $-\frac12<s\leq\frac{j-1}{2}$, then we have
\bee\label{th7--2}
|s_{2j}(e^{\frac{i\pi}{4}}\sqrt{\frac{\zeta}{2}})|+|b_{2j}(e^{\frac{i\pi}{4}}\sqrt{\frac{\zeta}{2}})|\leq C(1+\frac{1}{2s+1})\zeta^{j-2s-1}||q||_{H^s}^2
||q||_{l_1^2DU^2}^{2j-2},
\ene
and
\begin{align}\label{th7--1}
\begin{aligned}
\int_{1}^{\infty}&\zeta^{2s-j}(|s_{2j}(e^{\frac{i\pi}{4}}\sqrt{\frac{\zeta}{2}})|+|b_{2j}(e^{\frac{i\pi}{4}}\sqrt{\frac{\zeta}{2}})|)d\zeta\vspace{0.05in}\\
&\lesssim(1+\frac{1}{j-1-2s}+\frac{1}{(2s+1)^2})||q||_{H^s}^2||q||_{l_1^2DU^2}^{2j-2}.
\end{aligned}
\end{align}

\end{theorem}

\begin{proof}
For convenience, we define some symbols:
\begin{align}\no
\left\{\begin{aligned}
&\sum_{k}\cdot:=\sum_{j=0,k=2^j}^{\infty}\cdot,\quad q=\sum_kq_k,\vspace{0.05in}\\
&\hat{q}_1=\chi_{|\xi|<1}\hat{q},\quad\hat{q}_{<k}=\chi_{|\xi|<k}\hat{q},\quad\hat{q}_{k}=\chi_{k\leq|\xi|<2k}\hat{q}.
\end{aligned}\right.
\end{align}

According to Theorem \ref{th6}, we have
\bee
|s_{2j}(e^{\frac{i\pi}{4}}\sqrt{\frac{\zeta}{2}})|\lesssim\zeta^{j}\sum_{k_1\geq k_2}||q_{k_1}||_{l_{\zeta}^2DU^2}||q_{k_2}||_{l_{\zeta}^2DU^2}
||q_{\leq k_2}||^{2j-2}_{l_{\zeta}^2DU^2}.
\ene

Below, we will to investigate the classification for the above inequality. According to Lemma \ref{le6}, we have
\begin{align}\label{th7-1}
\begin{aligned}
&||q_{k}||_{l^2_{\zeta}DU^2}\lesssim k^{-s}\zeta^{-\frac12}||q||_{H^s},\quad if~1<k\leq\zeta\vspace{0.05in}\\
&||q_{k}||_{l^2_{\zeta}DU^2}\lesssim k^{-s-\frac12}||q||_{H^s},\quad if~k\geq\zeta\vspace{0.05in}\\
&||q_{<k}||_{l^2_{\zeta}DU^2}\lesssim||q_{<k}||_{l^2_1DU^2},\quad if~k\geq\zeta\vspace{0.05in}\\
&||q_{<k}||_{l^2_{\zeta}DU^2}\lesssim k^{\frac12}\zeta^{-\frac12}||q_{<k}||_{l^2_1DU^2},\quad if~k<\zeta.
\end{aligned}
\end{align}

Then ,we obtain
\bee
|s_{2j}(e^{\frac{i\pi}{4}}\sqrt{\frac{\zeta}{2}})|\lesssim\zeta^{j-2s-1}\sum_{k_1\geq k_2}C(\zeta,k_1,k_2)||q_{k_1}||_{H^s}||q_{k_2}||_{H^s}||q||_{l_1^2DU^2}^{2j-2}
\ene
where
\bee
C(\zeta,k_1,k_2)=\begin{cases}
(\frac{k_1}{\zeta})^{j-2s-1}(\frac{k_2}{k_1})^{j-s-1},\quad k_2\leq k_1\leq\zeta,\vspace{0.03in}\\
(\frac{\zeta}{k_1})^{s+\frac12}(\frac{k_2}{\zeta})^{j-s-1},\quad k_2\leq\zeta\leq k_1,\vspace{0.03in}\\
(\frac{\zeta}{k_1})^{s+\frac12}(\frac{\zeta}{k_2})^{s+\frac12},\quad \zeta\leq k_2\leq k_1.
\end{cases}
\ene

Then the critical coefficients are $\frac{1}{2s+1}$, so we get Eq.~(\ref{th7--2}). Moreover, by the Cauchy-Schwarz inequality and Schur's lemma, we get
\bee\no
\int_{1}^{\infty}\zeta^{2s-j}|s_{2j}(e^{\frac{i\pi}{4}}\sqrt{\frac{\zeta}{2}})|d\zeta\lesssim c||q||_{H^s}^2{l^2_1DU^2},
\ene
where
\bee\no
c=\max\{\sup_{k_1}\sum_{\zeta,k_2}C(\zeta,k_1,k_2),\sup_{k_2}\sum_{\zeta,k_1}C(\zeta,k_1,k_2)\}.
\ene

Then the critical coefficients are $\frac{1}{j-1-2s}, \frac{1}{(2s+1)^2}$, so we get Eq.~(\ref{th7--1}). Thus the proof is completed.
\end{proof}

\section{Asymptotic analysis of $b_4(\lambda)$ and $b_6(\lambda)$}

In this section, we will provide asymptotic expressions for $b_4(\lambda)$ and $b_6(\lambda)$ and some related conclusions. For the analysis of $s_2(\lambda)$, we use the Fourier transform method. Here we use the same technique to analyze $b_4(\lambda)$ and $b_6(\lambda)$. We recall that $b_4(\lambda)$ is given by
\begin{align}\label{b4-1}
\begin{aligned}
b_4(\lambda)&=-2\lambda^4XXYY\vspace{0.05in}\\
&=-2\lambda^4\int_{x_1<x_2<y_1<y_2}q(y_1)q^*(x_1)q(y_2)q^*(x_2)e^{2i\lambda^2(y_1+y_2-x_1-x_2)}dx_1dy_1dx_2dy_2\vspace{0.05in}\\
&=-\frac{\lambda^4}{2\pi^2}\int_{\mathbb{R}^4}\int_{x_1<x_2<y_1<y_2}e^{2i\lambda^2(y_1+y_2-x_1-x_2)+i(y_1\eta_1+y_2\eta_2-x_1\xi_1-x_2\xi_2)}\vspace{0.05in}\\
&\quad\times \hat{q}(\eta_1)\hat{q}(\eta_2)\hat{q}^*(\xi_1)\hat{q}^*(\xi_2)dx_1dy_1dx_2dy_2d\xi_1d\xi_2d\eta_1d\eta_2\vspace{0.05in}\\
&=-\frac{\lambda^4}{2\pi^2}\int_{\mathbb{R}^4}\left(\int_{x_1<x_2<y_1<y_2}e^{i(2\lambda^2+\eta_1)y_1+i(2\lambda^2+\eta_2)y_2-i(2\lambda^2+\xi_1)x_1-i(2\lambda^2+\xi_2)x_2}dx_1dx_2dy_1dy_2\right)\vspace{0.05in}\\
&\quad\times \hat{q}(\eta_1)\hat{q}(\eta_2)\hat{q}^*(\xi_1)\hat{q}^*(\xi_2)d\xi_1d\xi_2d\eta_1d\eta_2.
\end{aligned}
\end{align}

Let
\begin{align}\label{b4-2}
\begin{aligned}
K(\xi_1,\xi_2,\eta_1,\eta_2):&=\int_{x_1<x_2<y_1<y_2}e^{i(2\lambda^2+\eta_1)y_1+i(2\lambda^2+\eta_2)y_2-i(2\lambda^2+\xi_1)x_1-i(2\lambda^2+\xi_2)x_2}dx_1dx_2dy_1dy_2\vspace{0.05in}\\
&=\int_{-\infty}^{\infty}\int_{-\infty}^{y_2}\int_{-\infty}^{y_1}\int_{-\infty}^{x_2}e^{i(2\lambda^2+\eta_1)y_1+i(2\lambda^2+\eta_2)y_2-i(2\lambda^2+\xi_1)x_1-i(2\lambda^2+\xi_2)x_2}dx_1dx_2dy_1dy_2\vspace{0.05in}\\
&=-\frac{2i\pi}{(2\lambda^2+\xi_1)(4\lambda^2+\xi_1+\xi_2)(2\lambda^2-\eta_1+\xi_1+\xi_2)}\delta(\eta_1+\eta_2-\xi_1-\xi_2).
\end{aligned}
\end{align}

\begin{lemma}\label{le8}
We have the following identity.
\begin{align}\label{le8-1}
\begin{aligned}
b_4(\lambda)&=\frac{i}{2\pi}\int_{\xi_1+\xi_2=\eta_1+\eta_2}\frac{\lambda^4}{(2\lambda^2+\xi_1)(2\lambda^2+\eta_1)(2\lambda^2+\eta_2)}\vspace{0.05in}\\
&\quad\times \mathrm{Re}\left(\hat{q}(\eta_1)\hat{q}(\eta_2)\hat{q}^*(\xi_1)\hat{q}^*(\xi_2)\right)d\xi_1d\eta_1d\eta_2.
\end{aligned}
\end{align}

Suppose that $q$ is a Schwartz function. Then we have the following asymptotic series.
\bee
b_4(\lambda)\sim i\sum\limits_{j=2}^{\infty}H_{j4}\frac{\lambda^{2-2j}}{2^{j+1}},
\ene
where
\bee\no
H_{j4}=-\mathrm{Re}\left(i^j\sum\limits_{\alpha_1+\alpha_2+\alpha_3=j-2}(-1)^{\alpha_1}\int q^{(\alpha_2)}q^{(\alpha_3)}\overline{q^{(\alpha_3)}q}dx\right).
\ene
\end{lemma}

\begin{proof}
Substituting Eq.~(\ref{b4-2}) into Eq.~(\ref{b4-1}), yields
\begin{align}\no
\begin{aligned}
b_4(\lambda)&=-\frac{\lambda^4}{2\pi^2}\int_{\mathbb{R}^4}K(\xi_1,\xi_2,\eta_1,\eta_2)\hat{q}(\eta_1)\hat{q}(\eta_2)\hat{q}^*(\xi_1)\hat{q}^*(\xi_2)d\xi_1d\xi_2d\eta_1d\eta_2\vspace{0.05in}\\
&=\frac{i}{\pi}\int_{\mathbb{R}^4}\frac{\lambda^4\delta(\eta_1+\eta_2-\xi_1-\xi_2)}{(2\lambda^2+\xi_1)(4\lambda^2+\xi_1+\xi_2)(2\lambda^2-\eta_1+\xi_1+\xi_2)}\hat{q}(\eta_1)\hat{q}(\eta_2)\hat{q}^*(\xi_1)\hat{q}^*(\xi_2)d\xi_1d\xi_2d\eta_1d\eta_2\vspace{0.05in}\\
&=\frac{i}{2\pi}\int_{\mathbb{R}^3}\frac{\lambda^4}{(2\lambda^2+\xi_1)(2\lambda^2+\frac{\eta_1}{2}+\frac{\eta_2}{2})(2\lambda^2+\eta_2)}\hat{q}(\eta_1)\hat{q}(\eta_2)\hat{q}^*(\xi_1)\hat{q}^*(\eta_1+\eta_2-\xi_1)d\xi_1d\eta_1d\eta_2.
\end{aligned}
\end{align}

We note that
\begin{align}\no
\begin{aligned}
\frac12\left[\frac{1}{(2\lambda^2+\frac{\eta_1}{2}+\frac{\eta_2}{2})(2\lambda^2+\eta_1)}+\frac{1}{(2\lambda^2+\frac{\eta_1}{2}+\frac{\eta_2}{2})(2\lambda^2+\eta_2)}\right]
=\frac{1}{(2\lambda^2+\eta_1)(2\lambda^2+\eta_2)}.
\end{aligned}
\end{align}

We can take advantage of the symmetry between $\xi_1,\xi_2$ and $\eta_1,\eta_2$, then
\bee\no
b_4(\lambda)=\frac{i}{2\pi}\int_{\xi_1+\xi_2=\eta_1+\eta_2}\frac{\lambda^4}{(2\lambda^2+\xi_1)(2\lambda^2+\eta_1)(2\lambda^2+\eta_2)}\mathrm{Re}\left(\hat{q}(\eta_1)\hat{q}(\eta_2)\hat{q}^*(\xi_1)\hat{q}^*(\xi_2)\right)d\xi_1d\eta_1d\eta_2.
\ene

Expanding Eq.~(\ref{le8-1}) to the negative power, we have
\begin{align}\no
\begin{aligned}
b_4(\lambda)&\sim\frac{i}{2\pi}\int_{\xi_1+\xi_2=\eta_1+\eta_2}\frac{1}{8\lambda^2}\sum\limits_{j_1=0}^{\infty}(-\frac{\xi_1}{2\lambda^2})^{j_1}
\sum\limits_{j_2=0}^{\infty}(-\frac{\eta_1}{2\lambda^2})^{j_2}\sum\limits_{j_3=0}^{\infty}(-\frac{\eta_2}{2\lambda^2})^{j_3}\vspace{0.05in}\\
&\quad\times \mathrm{Re}\left(\hat{q}(\eta_1)\hat{q}(\eta_2)\hat{q}^*(\xi_1)\hat{q}^*(\xi_2)\right)d\xi_1d\eta_1d\eta_2.
\end{aligned}
\end{align}

Then, the corresponding coefficient of $i\frac{\lambda^{2-2j}}{2^{j+1}}$ as
\begin{align}\no
\begin{aligned}
H_{j4}:&=\frac{1}{2\pi}\mathrm{Re}\sum\limits_{\alpha_1+\alpha_2+\alpha_3=j-2}(-1)^j\int_{\xi_1+\xi_2=\eta_1+\eta_2}\xi_1^{\alpha_1}
\eta_1^{\alpha_2}\eta_2^{\alpha_3}\hat{q}^*(\xi_1)\hat{q}^*(\xi_2)\hat{q}(\eta_1)\hat{q}(\eta_2)d\xi_1d\eta_1d\eta_2\vspace{0.05in}\\
&=\frac{1}{2\pi}\mathrm{Re}\sum\limits_{\alpha_1+\alpha_2+\alpha_3=j-2}(-1)^ji^{2-j}\int_{\xi_1+\xi_2=\eta_1+\eta_2}(i\xi_1)^{\alpha_1}
(i\eta_1)^{\alpha_2}(i\eta_2)^{\alpha_3}\hat{q}^*(\xi_1)\hat{q}^*(\xi_2)\hat{q}(\eta_1)\hat{q}(\eta_2)d\xi_1d\eta_1d\eta_2\vspace{0.05in}\\
&=\frac{1}{2\pi}\mathrm{Re}\sum\limits_{\alpha_1+\alpha_2+\alpha_3=j-2}i^{j-2}(-1)^{\alpha_1}\widehat{\overline{q^{(\alpha_1)}}}*\widehat{\overline{q}}
*\widehat{q^{(\alpha_2)}}*\widehat{q^{(\alpha_3)}}(0)\vspace{0.05in}\\
&=-\mathrm{Re}\left(i^j\sum\limits_{\alpha_1+\alpha_2+\alpha_3=j-2}(-1)^{\alpha_1}\int q^{(\alpha_2)}q^{(\alpha_3)}\overline{q^{(\alpha_3)}q}dx\right).
\end{aligned}
\end{align}

Thus the proof is completed.
\end{proof}

As the same way, we will provide an asymptotic expression for $b_6(\lambda)$.
\begin{lemma}\label{le9}
We have the following identity.
\begin{align}\label{le9--1}
\begin{aligned}
b_6(\lambda)&=-\frac{i}{4\pi^2}\int_{\xi_1+\xi_2+\xi_3=\eta_1+\eta_2+\eta_3}\frac{\lambda^6}{(2\lambda^2+\xi_1)(2\lambda^2+\xi_2)(2\lambda^2+\eta_2)(2\lambda^2+\eta_3)}\vspace{0.05in}\\
&\quad\times \left(\frac{1}{2\lambda^2+\eta_1}+\frac{1}{2\lambda^2+\xi_1+\xi_2-\eta_1}\right)\vspace{0.05in}\\
&\quad\times \hat{q}(\xi_1)\hat{q}(\xi_2)\hat{q}(\xi_3)\hat{q}^*(\eta_1)\hat{q}^*(\eta_2)\hat{q}^*(\eta_3)d\xi_1d\xi_2d\eta_1d\eta_2\eta_3.
\end{aligned}
\end{align}

Suppose that $q$ is a Schwartz function. Then we have the following asymptotic series.
\bee\label{le9-0}
b_6(\lambda)\sim -i\sum\limits_{j=4}^{\infty}H_{j6}\frac{\lambda^{4-2j}}{2^{j+1}},
\ene
where
\begin{align}\no
\begin{aligned}
H_{j6}&=\mathrm{Re}\bigg(i^j\sum\limits_{\alpha_1+\alpha_2+\alpha_3+\alpha_4+\alpha_5=j-4}(-1)^{\alpha_1+\alpha_2}\int q^{(\alpha_1)}\vspace{0.05in}\\
&\quad \times q^{(\alpha_2)}q\overline{q^{(\alpha_3)}q^{(\alpha_4)}q^{(\alpha_5)}}+q^{(\alpha_1)}q^{(\alpha_2)}q^{*}(q\overline{q^{(\alpha_4)}q^{(\alpha_5)}})^{\alpha_3}dx\bigg).
\end{aligned}
\end{align}

\end{lemma}

\begin{proof}
According to Theorem \ref{th3}, we have
\bee\label{le9-1}
b_6(\lambda)=-4\lambda^6(XXYXYY+3XXXYYY).
\ene

We will calculate the two terms on the right side of Eq.~(\ref{le9-1}) respectively.
\begin{align}\label{le9-2}
\begin{aligned}
-4\lambda^6XXYXYY&=-4\lambda^6\int_{x_1<x_2<y_1<x_3<y_2<y_3}q(y_1)q^*(x_1)q(y_2)q^*(x_2)q(y_3)q^*(x_3)\vspace{0.05in}\\
&\quad \times e^{2i\lambda^2(y_1+y_2+y_3-x_1-x_2-x_3)}dx_1dy_1dx_2dy_2dx_3dy_3\vspace{0.05in}\\
&=-\frac{\lambda^6}{2\pi^3}\int_{\mathbb{R}^6}\left(\int_{x_1<x_2<y_1<x_3<y_2<y_3}e^{\sum\limits_{j=1}^{3}i(2\lambda^2+\eta_j)y_j-\sum\limits_{k=1}^{3}i(2\lambda^2+\xi_k)x_k}dx_1dx_2dx_3dy_1dy_2dy_3\right)\vspace{0.05in}\\
&\quad\times \hat{q}(\eta_1)\hat{q}(\eta_2)\hat{q}(\eta_3)\hat{q}^*(\xi_1)\hat{q}^*(\xi_2)\hat{q}^*(\xi_3)d\xi_1d\xi_2d\xi_3d\eta_1d\eta_2d\eta_3.
\end{aligned}
\end{align}

Let
\begin{align}\label{le9-3}
\begin{aligned}
K(\xi_1,\xi_2,\xi_3,\eta_1,\eta_2,\eta_3):&=\int_{x_1<x_2<y_1<x_3<y_2<y_3}e^{\sum\limits_{j=1}^{3}i(2\lambda^2+\eta_j)y_j-\sum\limits_{k=1}^{3}i(2\lambda^2+\xi_k)x_k}dx_1dx_2dx_3dy_1dy_2dy_3\vspace{0.05in}\\
&=\int_{-\infty}^{+\infty}\int_{-\infty}^{y_3}\int_{-\infty}^{y_2}\int_{-\infty}^{x_3}\int_{-\infty}^{y_1}\int_{-\infty}^{x_2}
e^{\sum\limits_{j=1}^{3}i(2\lambda^2+\eta_j)y_j-\sum\limits_{k=1}^{3}i(2\lambda^2+\xi_k)x_k}dx_1dx_2dx_3dy_1dy_2dy_3\vspace{0.05in}\\
&=2\pi i\frac{1}{2\lambda^2+\xi_1}\frac{1}{4\lambda^2+\xi_1+\xi_2}\frac{1}{2\lambda^2+\xi_1+\xi_2-\eta_1}\frac{1}{4\lambda^2+\xi_1+\xi_2+\xi_3-\eta_1}\vspace{0.05in}\\
&\quad \times \frac{1}{2\lambda^2+\xi_1+\xi_2+\xi_3-\eta_1-\eta_2}\delta(\eta_1+\eta_2+\eta_3-\xi_1-\xi_2-\xi_3).
\end{aligned}
\end{align}

Substituting Eq.~(\ref{le9-3}) into Eq.~(\ref{le9-2}), yields
\begin{align}
\begin{aligned}
-4\lambda^6XXYXYY&=-\frac{\lambda^6}{2\pi^3}\int_{\mathbb{R}^6}K(\xi_1,\xi_2,\xi_3,\eta_1,\eta_2,\eta_3)\vspace{0.05in}\\
&\quad\times \hat{q}(\eta_1)\hat{q}(\eta_2)\hat{q}(\eta_3)\hat{q}^*(\xi_1)\hat{q}^*(\xi_2)\hat{q}^*(\xi_3)d\xi_1d\xi_2d\xi_3d\eta_1d\eta_2d\eta_3\vspace{0.05in}\\
&=-\frac{i\lambda^6}{4\pi^2}\int_{\xi_1+\xi_2+\xi_3=\eta_1+\eta_2+\eta_3}\frac{1}{2\lambda^2+\xi_1}\frac{1}{2\lambda^2+\frac12\xi_1+\frac12\xi_2}\frac{1}{2\lambda^2+\xi_1+\xi_2-\eta_1}
\vspace{0.05in}\\
&\quad\times \frac{1}{2\lambda^2+\frac12\eta_2+\frac12\eta_3}\frac{1}{2\lambda^2+\eta_3} \hat{q}(\eta_1)\hat{q}(\eta_2)\hat{q}(\eta_3)\hat{q}^*(\xi_1)\hat{q}^*(\xi_2)\hat{q}^*(\xi_3)d\xi_1d\xi_2d\eta_1d\eta_2d\eta_3\vspace{0.05in}\\
&=-\frac{i\lambda^6}{4\pi^2}\int_{\xi_1+\xi_2+\xi_3=\eta_1+\eta_2+\eta_3}\frac{1}{2\lambda^2+\xi_1}\frac{1}{2\lambda^2+\xi_2}\frac{1}{2\lambda^2+\xi_1+\xi_2-\eta_1}
\frac{1}{2\lambda^2+\eta_2}\vspace{0.05in}\\
&\quad\times \frac{1}{2\lambda^2+\eta_3} \hat{q}(\eta_1)\hat{q}(\eta_2)\hat{q}(\eta_3)\hat{q}^*(\xi_1)\hat{q}^*(\xi_2)\hat{q}^*(\xi_3)d\xi_1d\xi_2d\eta_1d\eta_2d\eta_3.
\end{aligned}
\end{align}

On the other hand,
\begin{align}\label{le9-4}
\begin{aligned}
-12\lambda^6XXXYYY&=-12\lambda^6\int_{x_1<x_2<x_3<y_1<y_2<y_3}q(y_1)q^*(x_1)q(y_2)q^*(x_2)q(y_3)q^*(x_3)\vspace{0.05in}\\
&\quad \times e^{2i\lambda^2(y_1+y_2+y_3-x_1-x_2-x_3)}dx_1dy_1dx_2dy_2dx_3dy_3\vspace{0.05in}\\
&=-\frac{3\lambda^6}{2\pi^3}\int_{\mathbb{R}^6}\left(\int_{x_1<x_2<x_3<y_1<y_2<y_3}e^{\sum\limits_{j=1}^{3}i(2\lambda^2+\eta_j)y_j-\sum\limits_{k=1}^{3}i(2\lambda^2+\xi_k)x_k}dx_1dx_2dx_3dy_1dy_2dy_3\right)\vspace{0.05in}\\
&\quad\times \hat{q}(\eta_1)\hat{q}(\eta_2)\hat{q}(\eta_3)\hat{q}^*(\xi_1)\hat{q}^*(\xi_2)\hat{q}^*(\xi_3)d\xi_1d\xi_2d\xi_3d\eta_1d\eta_2d\eta_3
\end{aligned}
\end{align}

Let
\begin{align}\label{le9-5}
\begin{aligned}
H(\xi_1,\xi_2,\xi_3,\eta_1,\eta_2,\eta_3):&=\int_{x_1<x_2<x_3<y_1<y_2<y_3}e^{\sum\limits_{j=1}^{3}i(2\lambda^2+\eta_j)y_j-\sum\limits_{k=1}^{3}i(2\lambda^2+\xi_k)x_k}dx_1dx_2dx_3dy_1dy_2dy_3\vspace{0.05in}\\
&=\int_{-\infty}^{+\infty}\int_{-\infty}^{y_3}\int_{-\infty}^{y_2}\int_{-\infty}^{y_1}\int_{-\infty}^{x_3}\int_{-\infty}^{x_2}
e^{\sum\limits_{j=1}^{3}i(2\lambda^2+\eta_j)y_j-\sum\limits_{k=1}^{3}i(2\lambda^2+\xi_k)x_k}dx_1dx_2dx_3dy_1dy_2dy_3\vspace{0.05in}\\
&=2\pi i\frac{1}{2\lambda^2+\xi_1}\frac{1}{4\lambda^2+\xi_1+\xi_2}\frac{1}{6\lambda^2+\xi_1+\xi_2+\xi_3}\frac{1}{4\lambda^2+\xi_1+\xi_2+\xi_3-\eta_1}\vspace{0.05in}\\
&\quad \times \frac{1}{2\lambda^2+\xi_1+\xi_2+\xi_3-\eta_1-\eta_2}\delta(\eta_1+\eta_2+\eta_3-\xi_1-\xi_2-\xi_3).
\end{aligned}
\end{align}

Substituting Eq.~(\ref{le9-5}) into Eq.~(\ref{le9-4}) yields
\begin{align}
\begin{aligned}
-12\lambda^6XXXYYY&=-\frac{3\lambda^6}{2\pi^3}\int_{\mathbb{R}^6}H(\xi_1,\xi_2,\xi_3,\eta_1,\eta_2,\eta_3)\vspace{0.05in}\\
&\quad\times \hat{q}(\eta_1)\hat{q}(\eta_2)\hat{q}(\eta_3)\hat{q}^*(\xi_1)\hat{q}^*(\xi_2)\hat{q}^*(\xi_3)d\xi_1d\xi_2d\xi_3d\eta_1d\eta_2d\eta_3\vspace{0.05in}\\
&=-\frac{i\lambda^6}{4\pi^2}\int_{\xi_1+\xi_2+\xi_3=\eta_1+\eta_2+\eta_3}\frac{1}{2\lambda^2+\xi_1}\frac{1}{2\lambda^2+\frac12\xi_1
+\frac12\xi_2}\frac{1}{2\lambda^2+\frac13(\xi_1+\xi_2+\xi_3)}
\vspace{0.05in}\\
&\quad\times \frac{1}{2\lambda^2+\frac12\eta_2+\frac12\eta_3}\frac{1}{2\lambda^2+\eta_3} \hat{q}(\eta_1)\hat{q}(\eta_2)\hat{q}(\eta_3)\hat{q}^*(\xi_1)\hat{q}^*(\xi_2)\hat{q}^*(\xi_3)d\xi_1d\xi_2d\eta_1d\eta_2d\eta_3\vspace{0.05in}\\
&=-\frac{i\lambda^6}{4\pi^2}\int_{\xi_1+\xi_2+\xi_3=\eta_1+\eta_2+\eta_3}\frac{1}{2\lambda^2+\xi_1}\frac{1}{2\lambda^2+\xi_2}\frac{1}{2\lambda^2+\eta_1}
\frac{1}{2\lambda^2+\eta_2}\vspace{0.05in}\\
&\quad\times \frac{1}{2\lambda^2+\eta_3} \hat{q}(\eta_1)\hat{q}(\eta_2)\hat{q}(\eta_3)\hat{q}^*(\xi_1)\hat{q}^*(\xi_2)\hat{q}^*(\xi_3)d\xi_1d\xi_2d\eta_1d\eta_2d\eta_3,
\end{aligned}
\end{align}
since
\begin{align}\no
\begin{aligned}
&\frac13\bigg(\frac{1}{(2\lambda^2+\eta_1)(2\lambda^2+\eta_2)}+\frac{1}{(2\lambda^2+\eta_1)(2\lambda^2+\eta_3)}+
\frac{1}{(2\lambda^2+\eta_2)(2\lambda^2+\eta_3)}\bigg)\frac{1}{2\lambda^2
+\frac13(\eta_1+\eta_2+\eta_3)}\vspace{0.05in}\\
=&\frac{1}{(2\lambda^2+\eta_1)(2\lambda^2+\eta_2)(2\lambda^2+\eta_3)}.
\end{aligned}
\end{align}

Expanding Eq.~(\ref{le9--1}) to the negative power, we can obtain Eq.~(\ref{le9-0}). Thus the proof is completed.

\end{proof}

Based on the above analysis, we provide an asymptotic estimate of $b_{2j}$.
\begin{lemma}\label{le10}
The following estimate holds:
\bee\no
b_{2j}(\lambda)\sim\mathcal{O}(\lambda^{-2j+2}),\quad j\geq2.
\ene
\end{lemma}

\begin{proof}
Based on the properties of the Hopf algebra which we constructed earlier, we know that $b_{2j}(\lambda)'s$ are formal linear combinations of connected integrals. Then we obtain this lemma from the properties of connected integrals.
\end{proof}

We recall that
\bee
s_2(\lambda)=-\frac{i}{2}||q(x)||^2_{L^2}+i\int_{-\infty}^{+\infty}\frac{\xi}{4\lambda^2+2\xi}|\hat{q}(\xi)|^2d\xi.
\ene

According to Lemma \ref{le10}, we can obtain Eq.~(\ref{coro1-1}) and (\ref{coro1-2}). Then, Theorem \ref{1-th3} has been proven.

\section{Expansions for the iterative integrals $b_{2j}(\lambda)$ }

The overall properties of $s_{2j}(\lambda)$ and $b_{2j}(\lambda)$ were given in the previous section, and the properties of $b_{2j}(\lambda)$ need to be considered separately in this section.

\begin{lemma}\label{le11}
$b_{2j}$ have the following estimation:
\bee
|\lambda^{-2j}b_{2j}(\lambda)|\lesssim||e^{i\mathrm{Re}\lambda^2x}q||_{l_{\mathrm{Im}\lambda^2}^{2j}DU^2}^{2j}.
\ene
\end{lemma}

\begin{proof}
The proof is a direct consequence of Theorem \ref{th6}.
\end{proof}

\begin{theorem}\label{th75}
Suppose that $q\in H^s$. Then we have
\bee\no
|b_{2j}(e^{\frac{i\pi}{4}}\sqrt{\zeta})|\lesssim\zeta^{j-2s-1}||q||^2_{H^s}||q||_{l_{1}^{2}DU^2}^{2j-2},\quad s\leq j-1.
\ene
and
\begin{align}\no
\begin{aligned}
\int_{1}^{\infty}&\zeta^{2s-j}|b_{2j}(e^{\frac{i\pi}{4}}\sqrt{\frac{\zeta}{2}})|d\zeta
\lesssim2^{-j}(1+\frac{1}{j-1-s})||q||_{H^s}^2||q||_{l_1^2DU^2}^{2j-2},\quad 0\leq s<j-1.
\end{aligned}
\end{align}

\end{theorem}

\begin{proof}
Firstly, we have
\bee
|b_{2j}(e^{\frac{i\pi}{4}}\sqrt{\zeta})|\lesssim\zeta^{j}\sum_{k_1\geq k_2}||q_{k_1}||_{l_{\zeta}^{2j}DU^2}||q_{k_2}||_{l_{\zeta}^{2j}DU^2}
||q_{\leq k_2}||^{2j-2}_{l_{\zeta}^{2j}DU^2}.
\ene

A proof method similar to Theorem \ref{th7}, first of all, if $1\leq k<\zeta$, then we have
\bee\label{th75-3}
||q_k||_{l^{2j}_{\zeta}DU^2}\lesssim k^{\frac12-s-\frac{1}{2j}}\zeta^{\frac{1}{j}-2}||q||_{H^s}.
\ene
and
\bee\label{th75-4}
||q_{\leq k}||_{l_{\zeta}^{2j}DU^2}\lesssim k^{1-\frac{1}{2j}}\zeta^{\frac{1}{2j}-1}||q||_{l_{1}^{2}DU^2}.
\ene

If $k\geq\zeta$, then we have
\bee\label{th75-5}
||q_k||_{l^{2j}_{\zeta}DU^2}\lesssim k^{-s-\frac12}||q||_{H^s}.
\ene
and
\bee\label{th75-6}
||q_{\leq k}||_{l^{2j}_{\zeta}DU^2}\lesssim ||q||_{l^2_1DU^2}.
\ene

Then ,we obtain
\bee
|s_{2j}(e^{\frac{i\pi}{4}}\sqrt{\frac{\zeta}{2}})|\lesssim\zeta^{j-2s-1}\sum_{k_1\geq k_2}C(\zeta,k_1,k_2)||q_{k_1}||_{H^s}||q_{k_2}||_{H^s}||q||_{l_1^2DU^2}^{2j-2}
\ene
where
\bee
C(\zeta,k_1,k_2)=\begin{cases}
(\frac{k_1}{\zeta})^{2(j-s-1)}(\frac{k_2}{k_1})^{2j-s-\frac52+\frac{1}{2j}},\quad k_2\leq k_1\leq\zeta,\vspace{0.03in}\\
(\frac{\zeta}{k_1})^{s+\frac12}(\frac{k_2}{\zeta})^{2j-s-\frac52+\frac{1}{2j}},\quad k_2\leq\zeta\leq k_1,\vspace{0.03in}\\
(\frac{\zeta}{k_1})^{s+\frac12}(\frac{\zeta}{k_2})^{s+\frac12},\quad \zeta\leq k_2\leq k_1.
\end{cases}
\ene

Then, by the Cauchy-Schwarz inequality and Schur's lemma, we get the critical coefficient is $\frac{1}{j-1-s}$. Thus the proof is completed.

\end{proof}

Given $\Sigma_j$ a connected symbol of length $2j$, we will study the asymptotic expressions of the following iterated integral.
\bee\no
T_{\Sigma_j}(\lambda)=\lambda^{2j}\int_{\Sigma_j}\prod_{k=1}^{j}e^{2i\lambda^2(y_k-x_k)}q(y_k)q^{*}(x_k)dx_1dy_1\cdot\cdot\cdot dx_jdy_j.
\ene

\begin{theorem}\label{th9}
The connected integrals $T_{\Sigma_j}(\lambda)$ have the following asymptotic expressions.
\bee\no
T_{\Sigma_j}(\lambda)\sim\sum_{l=0}^{\infty}T_{\Sigma_j}^{l}2^{1-2j-l}\lambda^{-(2j-2+2l)},
\ene
where
\bee\no
T_{\Sigma_j}^{l}=\sum\limits_{|\alpha|+|\beta|=l}c_{\alpha\beta}\int\prod_{k=1}^{j}\partial^{\alpha_k}q^*_k\partial^{\beta_k}q_kdx,
\ene
with
\bee\no
c_{\alpha\beta}=\frac{1}{\alpha!\beta!}\int_{\Sigma_j,x_1=0}\prod e^{y_j-x_j}x_j^{\alpha}y_j^{\beta}dx_jdy_j.
\ene

and the errors in the above expansion have the following bounds:
\begin{align}\no
\begin{aligned}
&|T_{\Sigma_j}(\frac{e^{\frac{i\pi}{4}}\zeta}{\sqrt{2}})-\sum_{l=0}^{k}T_{\Sigma_j}^{l}2^{-j}i^{-(j-1+l)}\zeta^{-(2j-2+2l)}|\vspace{0.05in}\\
\lesssim&\sum\limits_{k+1\leq|\alpha|+|\beta|\leq2j-1+k}2^{-j}|\zeta^2|^{j-|\alpha|-|\beta|}\prod_{k}||\partial^{\alpha_k}q_k^*||_{l^{2j}_{\zeta^2}DU^2}
||\partial^{\beta_k}q_k||_{l^{2j}_{\zeta^2}DU^2}.
\end{aligned}
\end{align}
where $\max\{\alpha_k,\beta_k\}\leq[\frac{k}{2}]+1$ and $\zeta\geq1$.
\end{theorem}

\begin{proof}
Firstly, we have
\begin{align}\label{th9-1}
\begin{aligned}
T_{\Sigma_j}(\lambda)&=\lambda^{2j}\int_{\Sigma_j}\prod_{k=1}^{j}e^{2i\lambda^2(y_k-x_k)}q(y_k)q^{*}(x_k)dx_1dy_1\cdot\cdot\cdot dx_jdy_j\vspace{0.05in}\\
&=\lambda^{2j}\int_{\Sigma_j}\prod_{k=1}^{j}e^{2i\lambda^2(y_k-x_k)}\sum_{\beta_k=0}^{\infty}\frac{1}{\beta_k!}\partial^{\beta_k}q(x_1)(y_k-x_1)^{\beta_k}\vspace{0.05in}\\
&\quad\times \sum_{\alpha_k=0}^{\infty}\frac{1}{\alpha_k!}\partial^{\alpha_k}q^{*}(x_1)(x_k-x_1)^{\alpha_k}dx_1dy_1\cdot\cdot\cdot dx_jdy_j\vspace{0.05in}\\
&=\sum\limits_{l=0}^{\infty}\lambda^{2j}\int_{\Sigma_j}\sum\limits_{|\alpha|+|\beta|=l}\prod_{k=1}^{j}e^{2i\lambda^2(y_k-x_k)}\frac{1}{\beta_k!}\partial^{\beta_k}q(x_1)(y_k-x_1)^{\beta_k}\vspace{0.05in}\\
&\quad\times \frac{1}{\alpha_k!}\partial^{\alpha_k}q^{*}(x_1)(x_k-x_1)^{\alpha_k}dx_1dy_1\cdot\cdot\cdot dx_jdy_j\vspace{0.05in}\\
&=:\sum\limits_{l=0}^{\infty}\sum\limits_{|\alpha|+|\beta|=l}T_{\Sigma_j}^{\alpha\beta},
\end{aligned}
\end{align}
where $\alpha=(\alpha_1,\alpha_2,\cdot\cdot\cdot,\alpha_j),\beta=(\beta_1,\beta_2,\cdot\cdot\cdot,\beta_j),|\alpha|=\sum_{k=1}^{j}\alpha_k,
|\beta|=\sum_{k=1}^{j}\beta_k$.

For convenience, we redefine the notations:
\begin{align}\no
\begin{aligned}
&\{x_1,y_1,x_2,y_2,\cdot\cdot\cdot,x_j,y_j\}_{\Sigma_j}=\{t_1,t_2,t_3,t_4,\cdot\cdot\cdot,t_{2j-1},t_{2j}\},\vspace{0.05in}\\
&\{q,q^{*},\cdot\cdot\cdot,q,q^*\}_{\Sigma_j}=\{v_1,v_2,\cdot\cdot\cdot,v_{2j-1},v_{2j}\}.
\end{aligned}
\end{align}

Then, we note that
\begin{align}\label{th9-2}
\begin{aligned}
T_{\Sigma_j}^{\alpha\beta}&=\int_{\Sigma_j}\lambda^{2j}\prod_{k=1}^{j}e^{2i\lambda^2(y_k-x_k)}\frac{1}{\beta_k!}\partial^{\beta_k}q(x_1)(y_k-x_1)^{\beta_k}\vspace{0.05in}\\
&\quad\times \frac{1}{\alpha_k!}\partial^{\alpha_k}q^{*}(x_1)(x_k-x_1)^{\alpha_k}dx_1dy_1\cdot\cdot\cdot dx_jdy_j\vspace{0.05in}\\
&=\int_{-\infty}^{+\infty}\int_{t_1}^{+\infty}\int_{t_2}^{+\infty}\cdot\cdot\cdot\int_{t_{2j-1}}^{+\infty}\lambda^{2j}\prod_{k=1}^{j-1}e^{2i\lambda^2(y_k-x_k)}\frac{1}{\beta_k!}\partial^{\beta_k}q(t_1)(y_k-t_1)^{\beta_k}\vspace{0.05in}\\
&\quad\times\frac{1}{\alpha_k!}\partial^{\alpha_k}q^{*}(t_1)(x_k-t_1)^{\alpha_k}e^{-2i\lambda^2x_{j}}\frac{1}{\alpha_{j}!}\partial^{\alpha_{j}}q^{*}(t_1)(x_{j}-t_1)^{\alpha_{j}}\vspace{0.05in}\\
&\quad\times e^{2i\lambda^2t_{2j}}\frac{1}{\beta_j!}\partial^{\beta_j}q(t_1)(t_{2j}-t_1)^{\beta_j}
dt_1dt_2\cdot\cdot\cdot dt_{2j-1}dt_{2j}.
\end{aligned}
\end{align}

We first calculate the integral $T_{\Sigma_j}^{\alpha\beta}$ about $t_{2j}$, and we have
\begin{align}\label{th9-3}
\begin{aligned}
&\int_{t_{2j-1}}^{+\infty}e^{2i\lambda^2t_{2j}}\frac{1}{\beta_j!}\partial^{\beta_j}q(t_1)(t_{2j}-t_1)^{\beta_j}dt_{2j}\vspace{0.05in}\\
=&-\frac{1}{2i\lambda^2}e^{2i\lambda^2t_{2j-1}}\frac{1}{\beta_j!}\partial^{\beta_j}q(t_1)(t_{2j-1}-t_1)^{\beta_j}\vspace{0.05in}\\
&-\int_{t_{2j-1}}^{+\infty}\frac{1}{2i\lambda^2}e^{2i\lambda^2t_{2j}}\frac{1}{(\beta_j-1)!}\partial^{\beta_j}q(t_1)
(t_{2j}-t_1)^{\beta_j-1}dt_{2j}\vspace{0.05in}\\
=&\sum_{k=1}^{\beta_j}(-1)^k\frac{1}{(2i\lambda^2)^k}e^{2i\lambda^2t_{2j-1}}\frac{1}{(\beta_j-k+1)!}\partial^{\beta_j}q(t_1)(t_{2j-1}-t_1)^{\beta_j-k+1}\vspace{0.05in}\\
&+(-1)^{\beta_j}\int_{t_{2j-1}}^{+\infty}\frac{1}{(2i\lambda^2)^{\beta_j}}e^{2i\lambda^2t_{2j}}\partial^{\beta_j}q(t_1)dt_{2j}\vspace{0.05in}\\
=&\sum_{k=1}^{\beta_j+1}(-1)^k\frac{1}{(2i\lambda^2)^k}e^{2i\lambda^2t_{2j-1}}\frac{1}{(\beta_j-k+1)!}\partial^{\beta_j}q(t_1)(t_{2j-1}-t_1)^{\beta_j-k+1}.
\end{aligned}
\end{align}

Thus, the corresponding terms in the errors are linear combinations of the following integrals:
\begin{align}\no
\begin{aligned}
R=&\frac{1}{2^j\zeta^{2j-2+2l}}\int_{t_1=\cdot\cdot\cdot=t_{j_-+1}<\cdot\cdot\cdot<t_{2j-j_+}=\cdot\cdot\cdot
=t_{2j}}e^{\zeta^2\sum\limits_{i=1}^jx_i-y_i}
%\vspace{0.05in}\\&\times
\prod_{i=1}^{2j}\partial^{\alpha_i}v_i(t_i)dt_{j_-+1}\cdot\cdot\cdot dt_{2j-j_+},
\end{aligned}
\end{align}
where
\begin{align}\no
\begin{aligned}
&0\leq j_-\leq \alpha_-=[\frac{k+2}{2}],\quad 0\leq j_+\leq \alpha_+=k+1-[\frac{k+2}{2}],\quad j_-+j_+\leq 2j-2,\vspace{0.05in}\\
&\sum\limits_{i=1}^{1+j_-}\alpha_i=[\frac{k+1}{2}],\quad \sum\limits_{i=2j-j_+}^{2j}\alpha_i=[\frac{k+2}{2}].
\end{aligned}
\end{align}

Similar to the proof method of Theorem \ref{th6}, we have
\bee\no
|R|\lesssim\frac{1}{2^j\zeta^{2j-2+2l}}||v_-||_{l^{\frac{2j}{j_-+1}}_{\zeta^2}DU^2}
\prod_{i=2+j_-}^{2j-j_+-1}||v_i||_{l^{2j}_{\zeta^2}DU^2}||v_+||_{l^{\frac{2j}{j_++1}}_{\zeta^2}DU^2},
\ene
where
\bee\no
v_-=\prod_{i=1}^{1+j_-}\partial^{\alpha_i}v_i,\quad v_-=\prod_{i=2j-j_+}^{2j}\partial^{\alpha_i}v_i.
\ene

We note that
\begin{align}\no
\begin{aligned}
&||q_1q_2||_{l_{\zeta^2}^pDU^2}\lesssim\bigg|\bigg|||\chi_[\frac{k}{\zeta^2},\frac{k+1}{\zeta^2})q_1||_{V^2}\bigg|\bigg|_{l^q}
\bigg|\bigg|||\chi_[\frac{k}{\zeta^2},\frac{k+1}{\zeta^2})q_2||_{DU^2}\bigg|\bigg|_{l^r},\quad \frac{1}{p}=\frac{1}{q}+\frac{1}{r}\vspace{0.05in}\\
&||q||_{l^q_{\zeta^2}V^2}\lesssim ||q^{'}||_{l^qDU^2}+\zeta^2||q||_{l^qDU^2},
\end{aligned}
\end{align}
where $p\geq2$.

Then we have the following estimate
\bee\no
||v_-||_{l^{\frac{2j}{j_-+1}}_{\zeta^2}DU^2}\lesssim||\partial^{\alpha_1}v_1||_{l^{2j}DU^2}\prod_{i=2}^{j_--1}(
||\partial^{\alpha_i+1}v_i||_{l^{2j}DU^2}+\frac{1}{\zeta^2}||\partial^{\alpha_i}v_i||_{l^{2j}DU^2}).
\ene

We argue similarly for $v_+$. Thus the proof is completed.
\end{proof}

Based on the above analysis, we obtain the following Corollary.
\begin{corollary}\label{coro2}
The following estimate holds:
\begin{align}\no
\begin{aligned}
&|T_{\Sigma_j}(e^{\frac{i\pi}{4}}\sqrt{\frac{\zeta}{2}})-\sum_{l=0}^{k}T_{\Sigma_j}^{l}2^{-j}(i\zeta)^{-(j-1+l)}|\vspace{0.05in}\\
&\quad \lesssim\sum\limits_{k+1\leq|\alpha|+|\beta|\leq2j-1+k}2^{-j}\zeta^{j-|\alpha|-|\beta|}\prod_{k}||\partial^{\alpha_k}q_k^*||_{l^{2j}_{\zeta^2}DU^2}
||\partial^{\beta_k}q_k||_{l^{2j}_{\zeta^2}DU^2}.
\end{aligned}
\end{align}
where $\max\{\alpha_k,\beta_k\}\leq[\frac{k}{2}]+1$ and $\zeta\geq1$.
\end{corollary}

\begin{theorem}\label{1-th1}
Let $q(x)\in H^s(\mathbb{R})$ and $j-1+\frac{k_1}{2}\leq s\leq j-1+\frac{k_1+1}{2}\, (j,\, k_1\in \mathbb{Z}^+)$. Define the following iterated integral:
\bee\no
T_{\Sigma_j}(\lambda)=\lambda^{2j}\int_{\Sigma_j}\prod_{k=1}^{j}e^{2i\lambda^2(y_k-x_k)}q(y_k)q^{*}(x_k)dx_1dy_1\cdot\cdot\cdot dx_jdy_j,
\ene
and
\bee\no
T_{\Sigma_j}^{l}=\sum\limits_{|\alpha|+|\beta|=l}c_{\alpha\beta}\int\prod_{k=1}^{j}\partial^{\alpha_k}q^*_k\partial^{\beta_k}q_kdx,
\ene
with
\bee\no
c_{\alpha\beta}=\frac{1}{\alpha!\beta!}\int_{\Sigma_j,x_1=0}\prod e^{y_j-x_j}x_j^{\alpha}y_j^{\beta}dx_jdy_j.
\ene
Then the following error estimates hold:
\bee\no
|T_{\Sigma_j}(e^{\frac{i\pi}{4}}\sqrt{\frac{\zeta}{2}})-\sum_{l=0}^{k_1}T_{\Sigma_j}^{l}2^{-j}(i\zeta)^{-(j-1+l)}|
\lesssim2^{-j}\zeta^{j-2s-1}||q||^2_{H^s}||q||^{2j-2}_{l_1^2DU^2},
\ene
and
\begin{align}\no
\begin{aligned}
&\int_{1}^{+\infty}2^j\zeta^{2s-j}|T_{\Sigma_j}(e^{\frac{i\pi}{4}}\sqrt{\frac{\zeta}{2}})
-\sum_{l=0}^{k_1}T_{\Sigma_j}^{l}2^{-j}(i\zeta)^{-(j-1+l)}|d\zeta
%\vspace{0.05in}\\&
\lesssim \frac{1}{|\sin(2\pi s)|}||q||^2_{H^s}||q||^{2j-2}_{l_1^2DU^2},
\end{aligned}
\end{align}
where $\Sigma_j$ is an appropriate domain which obeys $x_k<y_k$ for all $k~(k\leq j)$.
%Moreover, let
%\bee\no
%E_s:=-\frac{2\sin(\pi s)}{\pi}\int_{1}^{\infty}(\zeta^2-1)^s[\mathrm{Re}\ln I(e^{\frac{i\pi}{4}}\sqrt{\frac{\zeta}{2}})-\sum\limits_{j=0}^M(-1)^j
%H_{2j}\zeta^{-2j-1}]d\zeta+\sum\limits_{j=0}^MC_{s}^jH_{2j},
%\ene
%where
%\bee\no
%H_{2j}=\frac{1}{\pi}\int_{-\infty}^{+\infty}\zeta^k\mathrm{Re}\ln I(\sqrt{\frac{\zeta}{2}})d\zeta.
%\ene
%then $E_s$ is conserved along the DNLS flow if $s\geq\frac12$.

\end{theorem}

\begin{proof} According to Corollary \ref{coro2}, we can show this Theorem.
\end{proof}

\section{Conclusions and Discussions}

We prove the well-posedness results of scattering data for the derivative nonlinear Schr\"odinger equation in the $H^{s}(\mathbb{R})(s\geq\frac12)$. We show that $s_{11}(\lambda)$ can be written as the sum of some iterative integrals, and its logarithm $\ln s_{11}(\lambda)$ can be written as the sum of some connected iterative integrals. And we provide the asymptotic properties of the first few iterative integrals of $s_{11}(\lambda)$. Moreover, we provide some regularity properties of $s_{11}(\lambda)$ related to scattering data in $H^{s}(\mathbb{R})$.

\vspace{0.2in}
\noindent {\bf Acknowledgments}

\vspace{0.05in}
This work was supported by the National Natural Science Foundation of China (Grant No. 11925108)
% and 11971475).

\end{document}